\documentclass[reqno,12pt]{amsart}
\usepackage[letterpaper, margin=1in]{geometry}
\RequirePackage{amsmath,amssymb,amsthm,graphicx,mathrsfs,url,slashed}
\RequirePackage[usenames,dvipsnames]{xcolor}
\RequirePackage[colorlinks=true,linkcolor=Red,citecolor=Green]{hyperref}
\RequirePackage{amsxtra}
\usepackage{cancel}
\usepackage{bbm}
\usepackage{tikz-cd}

\title[Higher-dimensional fractal uncertainty principle]{The fractal uncertainty principle via Dolgopyat's method in higher dimensions}
\author{Aidan Backus}
\address{Aidan Backus, Department of Mathematics, Brown University, Providence, RI}
\email{aidan\_backus@brown.edu}
\author{James Leng}
\address{James Leng, Department of Mathematics, University of California, Los Angeles, Los Angeles, CA}
\email{jamesleng@math.ucla.edu}
\author{Zhongkai Tao}
\address{Zhongkai Tao, Department of Mathematics, University of California, Berkeley, Berkeley, CA}
\email{ztao@math.berkeley.edu}
\date{\today}

\newcommand{\NN}{\mathbf{N}}
\newcommand{\ZZ}{\mathbf{Z}}

\newcommand{\RR}{\mathbf{R}}
\newcommand{\CC}{\mathbf{C}}

\newcommand{\Ball}{\mathbf{B}}

\DeclareMathOperator*{\Expect}{\mathbf E}
\DeclareMathOperator*{\Var}{\mathrm{Var}}
\DeclareMathOperator{\Acal}{\mathcal A}

\newcommand*\dif{\mathop{}\!\mathrm{d}}

\DeclareMathOperator{\diam}{diam}

\DeclareMathOperator{\argmax}{arg\,max}

\DeclareMathOperator{\id}{id}

\DeclareMathOperator{\supp}{supp}

\newcommand{\dist}{\mathrm{dist}_\infty}

\newcommand{\dfn}[1]{\emph{#1}\index{#1}}

\renewcommand{\Re}{\operatorname{Re}}
\renewcommand{\Im}{\operatorname{Im}}

\newtheorem{theorem}{Theorem}[section]

\newtheorem{prop}[theorem]{Proposition}
\newtheorem{lemma}[theorem]{Lemma}

\newtheorem{invariant}[theorem]{Invariant}

\newtheorem{proposition}[theorem]{Proposition}
\newtheorem{corollary}[theorem]{Corollary}

\theoremstyle{definition}
\newtheorem{definition}[theorem]{Definition}
\newtheorem{defi}[theorem]{Definition}
\newtheorem{remark}[theorem]{Remark}
\newtheorem{example}[theorem]{Example}



\definecolor{green}{rgb}{0,0.8,0} 

\numberwithin{equation}{section}

\def\XXint#1#2#3{{\setbox0=\hbox{$#1{#2#3}{\int}$ }
\vcenter{\hbox{$#2#3$ }}\kern-.6\wd0}}

\usepackage[
    backend=biber,
    style=alphabetic,
    giveninits=true
]{biblatex}
\DeclareFieldFormat{pages}{#1}
\renewbibmacro{in:}{%
  \ifentrytype{article}
    {}
    {\bibstring{in}%
     \printunit{\intitlepunct}}}
\DeclareFieldFormat
  [article,inbook,incollection,inproceedings,patent,thesis,unpublished]
  {title}{\mkbibemph{#1}}
\DeclareFieldFormat{journaltitle}{#1\isdot}
\DeclareFieldFormat[article]{volume}{\mkbibbold{#1}}
\DeclareFieldFormat[article]{number}{\bibstring{number}\addnbspace #1}

\renewbibmacro*{journal+issuetitle}{%
  \usebibmacro{journal}%
  \setunit*{\addspace}%
  \iffieldundef{series}
    {}
    {\newunit
     \printfield{series}%
     \setunit{\addspace}}%
  \printfield{volume}%
  \setunit{\addspace}%
  \usebibmacro{issue+date}%
  \setunit{\addcomma\space}%
  \printfield{number}%
  \setunit{\addcolon\space}%
  \usebibmacro{issue}%
  \setunit{\addcomma\space}%
  \printfield{eid}
  \newunit}

\begin{document}
\begin{abstract}
     We prove a fractal uncertainty principle with exponent $\frac{d}{2} - \delta + \varepsilon$, $\varepsilon > 0$, for Ahlfors--David regular subsets of $\RR^d$ with dimension $\delta$ which satisfy a suitable ``nonorthogonality condition". This generalizes the application of Dolgopyat's method by Dyatlov--Jin \cite{Dyatlov_2018} to prove the same result in the special case $d = 1$. As a corollary, we get a quantitative spectral gap for the Laplacian on convex cocompact hyperbolic manifolds of arbitrary dimension with Zariski dense fundamental groups.
\end{abstract}


\maketitle



\section{Introduction}
The \dfn{fractal uncertainty principle}, informally, is the assertion that a function cannot be microlocalized to a neighborhood of a fractal set in phase space.
Such assertions have applications in spectral theory, where one can apply microlocal methods to show that fractal uncertainty principles imply the existence of essential spectral gaps \cite{Dyatlov_2016}.
In particular, one can obtain $L^2 \to L^2$ bounds on the scattering resolvents of the Laplacian on convex cocompact hyperbolic manifolds, as well as improvements on the size of the maximal region in which certain zeta functions admit analytic continuation \cite{Bourgain_2018}.

To make the fractal uncertainty principle more precise, we introduce the \dfn{semiclassical Fourier transform}
$$\mathscr F_h f(\xi) := (2\pi h)^{-d/2} \int_{\RR^d} e^{-ix \cdot \xi/h} f(x) \dif x$$
where $h > 0$ is a small parameter.
If we have sets $X, Y$, and we write $X_h, Y_h$ for the sumsets $X_h := X + B_h$, $Y_h := Y + B_h$, $B_h := B(0, h)$, then the fractal uncertainty principle for $X, Y$ asserts bounds of the form 
\begin{equation}\label{FUP dfn}
\|1_{X_h} \mathscr F_h 1_{Y_h}\|_{L^2 \to L^2} \lesssim h^\beta
\end{equation}
in the limit $h \to 0$.
We will be interested in the case that $X, Y$ are Ahlfors--David regular sets:

\begin{definition}
A compactly supported finite Borel measure $\mu$ on $\RR^d$ is \dfn{Ahlfors--David regular} of dimension $\delta \in [0, d]$, on scales $[\alpha, \beta]$, with regularity constant $C_R \geq 1$, if for every closed square box $I$ with side length $r \in [\alpha, \beta]$, or closed ball $I$ with radius $r \in [\alpha, \beta]$,
$$\mu(I) \leq C_R r^\delta,$$
and if in addition $I$ is centered on a point in $X := \supp \mu$,
$$C_R^{-1} r^\delta \leq \mu(I).$$
In short we say that $(X, \mu)$ is \dfn{$\delta$-regular}.
\end{definition}

Applying Plancherel's theorem and H\"older's inequality, one can easily check that if $X$ is $\delta$-regular and $Y$ is $\delta'$-regular on scales $[h, 1]$, then 
\begin{equation}\label{noob bound}
\|1_{X_h} \mathscr F_h 1_{Y_h}\|_{L^2 \to L^2} \lesssim h^{\max\left(0, \frac{d - \delta - \delta'}{2}\right)};
\end{equation}
this estimate is a straightforward modification of \cite[(2.7)]{Dyatlov_2019}.
In fact, (\ref{noob bound}) is sharp if $\delta$ or $\delta'$ are either $0$ or $d$, or if $X, Y$ are orthogonal line segments in $\RR^2$.

Thus we say that $X, Y$ satisfy the \dfn{fractal uncertainty principle} if (\ref{FUP dfn}) holds for some $\beta > \max(0, \frac{d - \delta - \delta'}{2})$.
There are several cases in which the fractal uncertainty principle is known:
\begin{enumerate}
    \item If $d = 1$ and $0 < \delta, \delta' < 1$, then the fractal uncertainty principle holds \cite{Dyatlov_2016, Bourgain_2018, Dyatlov_2018}.
    \item If $d < \delta + \delta' < 2d$, then the fractal uncertainty principle holds under the additional assumption that either $Y$ can be decomposed as a product of Ahlfors--David fractals in $\RR$ \cite{Han_2020} or $Y$ is line-porous \cite{cohen2023fractal}.
    \item If $d$ is odd and $\delta, \delta'$ are very close to $d/2$, then the fractal uncertainty principle holds \cite{Cladek_2020}. 
    \item If $X,Y$ are arithmetic Cantor sets\footnote{We define these fundamental examples in \S\ref{model problem}, but for now the reader may view them as Cantor sets where the removed boxes have rational vertices.}, then the fractal uncertainty principle holds for $d=1$ \cite{Dyatlov_2017} and $d=2$, $\delta + \delta' \geq 1$ under the condition that $X$ does not contain any line \cite{cohen2022fractal}.
\end{enumerate}

\subsection{The main theorem}
In this paper we establish the fractal uncertainty principle for $0 < \delta + \delta' \leq d$ under the following additional hypothesis which rules out the possibility that $X, Y$ are orthogonal line segments.
For $\Phi(x, y) := -x \cdot y$ it is a quantitative form of the statement that ``$X$ and $Y$ do not lie in submanifolds which have orthogonal tangent spaces."

\begin{definition}\label{def:nonorthogonality}
Let $X,Y\subseteq \RR^d$ and let $\Phi \in C^2(\RR^d \times \RR^d)$.
We say that $(X,Y)$ is $\Phi$-\dfn{nonorthogonal} with constant $0 < c_N \leq 1$ from scales $(\alpha_0^X,\alpha_0^Y)$ to $(\alpha_1^X,\alpha_1^Y)$ if for any $x_0\in X$, $y_0\in Y$, and $r_X \in (\alpha_0^X, \alpha_1^X)$ and $r_Y \in (\alpha_0^Y, \alpha_1^Y)$, there exists $x_1, x_2\in X\cap B(x_0,r_X), y_1, y_2\in Y\cap B(y_0,r_Y)$ such that \begin{align}\label{nonorthogonality formula}
    |\Phi(x_1, y_1) - \Phi(x_2, y_1) - \Phi(x_1, y_2) + \Phi(x_2, y_2)| \geq c_N r_X r_Y.
\end{align}
\end{definition}

The motivation for this definition is as follows: we want nonorthogonality to be visible on virtually all scales; after all, orthogonality of fractals is a local property, so we want non-orthogonal examples on most balls centered on a point in $X$ and $Y$. The Ahlfors--David regularity condition guarantees that each such ball contributes roughly the same amount of fractal mass, and is hence the reason why we upgrade ``most" to ``all". At the same time, we don't want nonorthogonal points to lie too close to each other. This is why we take the right hand side to be $r_Xr_Y$ instead of $|x_1 - x_0| \cdot |y_1 - y_0|$. One can verify that this definition of nonorthogonality generalizes the nonorthogonality hypothesis of \cite[Proposition 6.5]{Dyatlov_2019}.

The nonorthogonality condition (\ref{nonorthogonality formula}) is based on the \dfn{local nonintegrability condition} (LNI) of \cite{Naud2005, stoyanov2011spectra}, which itself can be traced back to the \dfn{uniform nonintegrability condition} of \cite{Chernov98, Dolgopyat_98}.
In such papers one is concerned with the nonintegrability of the stable and unstable foliations of an Axiom A (or perhaps even Anosov) flow.
Roughly speaking, given fractals $X, Y$ one may define two laminations (in the sense of Thurston \cite[Chapter 8]{thurston1979geometry}) in $\RR^d_x \times \RR^d_\xi$, the \dfn{vertical lamination} $\{x \in X\}$ and \dfn{horizontal lamination} $\{\xi = \partial_x \Phi(x, y): y \in Y\}$, and then (\ref{nonorthogonality formula}) essentially asserts that the vertical and horizontal laminations satisfy LNI.

\begin{definition}\label{def:doubling measure}
A measure $\mu$ is \dfn{doubling} on scales $[h, 1]$ if there exists $C_D > 0$ such that for every $r \in [h, \frac{1}{2}]$ and every cube $I$ of side length $r$ centered at $x\in \supp\mu$, $\mu(I \cdot 2) \leq C_D \mu(I)$.
\end{definition}

Clearly every regular measure is doubling; we highlight that our main theorem only needs to assume doubling rather than regular. It is essential that we only consider cubes centered at $x\in \supp\mu$ in the definition. One can compare this doubling property with the Federer property in \cite[\S7]{Dolgopyat_98}, in which case the Gibbs measure is supported everywhere.

What follows is our main theorem:

\begin{theorem}\label{main theorem}
Let $\mu_X, \mu_Y$ be doubling probability measures on scales $[h, 1]$ with compact supports $X\subset I_0, Y\subset J_0$ where $I_0, J_0\subset\RR^d$ are rectangular boxes with unit length.
Let $\mathcal B_h$ be the semiclassical Fourier integral operator 
\begin{align}\label{semiclassical FIO definition}
    \mathcal{B}_h f(x)=\int_Y\exp\left(\frac{i\Phi(x,y)}{h}\right) p(x,y)f(y)\dif\mu_Y(y)
\end{align}
where the phase $\Phi \in C^3(I_0 \times J_0)$, $X, Y$ are $\Phi$-nonorthogonal from scales $h$ to $1$, and the symbol $p \in C^1(I_0 \times J_0)$.
Then there exists $\varepsilon_0 > 0$ such that
$$\|\mathcal B_h\|_{L^2(\mu_Y) \to L^2(\mu_X)} \lesssim h^{\varepsilon_0}.$$
\end{theorem}

If one additionally assumes $d = 1$, and that $\mu_X, \mu_Y$ are regular with dimension $\in (0, 1)$, then Theorem \ref{main theorem} was proven by Dyatlov--Jin \cite{Dyatlov_2018}, extending the method of Dolgopyat \cite{Dolgopyat_98} which had already been applied to construct spectral gaps. Using the construction of dyadic cubes in \cite{christ1990b}, it might be possible that Theorem \ref{main theorem} can be generalized to doubling metric spaces. Since there is no immediate application for metric spaces, we have not attempted to write down the more general version.

Following the methods of \cite{Dyatlov_2018}, Theorem \ref{main theorem} implies the following fractal uncertainty principle:

\begin{corollary}\label{FUP classic}
Let $X$ and $Y$ be Ahlfors--David regular sets in $\RR^d$, which are nonorthogonal with respect to the dot product on $\RR^d \times \RR^d$.
Assume that $X$ is $\delta$-regular, $Y$ is $\delta'$-regular, $0 < \delta, \delta' < d$. 
Then there exists $\varepsilon_0 > 0$ such that 
$$\|1_{X_h} \mathscr F_h 1_{Y_h}\|_{L^2 \to L^2} \lesssim h^{\frac{d - \delta - \delta'}{2} + \varepsilon_0}.$$
\end{corollary}

\subsubsection{Lower bounds on the uncertainty exponent}
If we let
\begin{equation}\label{L}
L := \frac{10^{14} d^3}{c_N^3} \max(1, \|\partial^2_{xy} \Phi\|_{C^1}^3)
\end{equation}
then we can take in Theorem \ref{main theorem}
\begin{equation}\label{epsilon0}
\frac{1}{\varepsilon_0} \leq 6 \cdot 10^9 c_N^{-2} d^2(C_D(X) C_D(Y))^{4\lceil \log_2(20L^{5/3}) \rceil} L^{2/3} \log L.
\end{equation}
In the model case that $X = Y$ is regular, $d = 1$, and $\Phi(x, y) = -xy$, we can always take $c_N = C_R^{-\frac{4}{\delta}}$ and $C_D = 2^\delta C_R^2$, which gives a subexponential bound of the form $1/\varepsilon \lesssim e^{C(\delta) \log_2 C_R}$. This is because of the rather poor dependence of $\varepsilon_0$ on the doubling constant; if one modified our proof to use the Ahlfors--David regularity directly, they would obtain a bound of the form $1/\varepsilon \lesssim C_R^{O(1 + 1/\delta)}$, which is comparable with the bound
$1/\varepsilon_0 \lesssim C_R^{\frac{160}{\delta(1 - \delta)}}$ of \cite{Dyatlov_2018}.

In any case, it does not seem that one can use Dolgopyat's method to obtain sharp fractal uncertainty principles, which therefore remains an interesting and challenging open problem. To drive this point home, we recall that in the case $d = 1$, $\delta = 1/2$, an unpublished manuscript of Murphy claims $1/\varepsilon_0 \lesssim \log C_R \log \log C_R$ \cite[\S1]{Cladek_2020}.

\subsubsection{Applications to spectral gaps}
Suppose $M=\Gamma\backslash \mathbf{H}^{d+1}$ is a (noncompact) convex cocompact hyperbolic manifold and $\Lambda(\Gamma)$ is the limit set (see \S{\ref{s:covcocom}} for the definition). The Patterson--Sullivan measure $\mu$ on $\Lambda(\Gamma)$ is Ahlfors--David regular of dimension $\delta_\Gamma\in [0,d)$ \cite[Theorem 7]{Sullivan1979}.
Under the condition that $\Gamma$ is Zariski dense\footnote{We note carefully that all varieties in this paper are considered to be over $\RR$, even when they have a natural structure as a $\CC$-variety!} in the algebraic group $SO(d + 1, 1)_0$, $(\Lambda(\Gamma),\mu)$ satisfies the nonorthogonality condition \eqref{nonorthogonality formula} for very general $\Phi(x,y)$ (see Corollary \ref{c:nonortho_hyper}). So we have the fractal uncertainty principle for $\Lambda(\Gamma)$ with very general phase functions.

Dyatlov--Zahl \cite{Dyatlov_2016} showed that fractal uncertainty principles can be used to prove essential spectral gaps. Let $\Delta$ be the Laplace--Beltrami operator on $M$, then the resolvent
$$R(\lambda):= \left(-\Delta-\frac{d^2}{4}-\lambda^2\right)^{-1}:L^2_{\rm comp}(M)\to H^2_{\rm loc}(M)$$
is well-defined for $\Im(\lambda)\gg 1$ and has a meromorphic continuation to $\lambda\in\CC$; see \cite{Mazzeo87,guillarmou2005meromorphic} for (even) asymptotically hyperbolic manifolds and \cite{guillope1995polynomial} for manifolds with constant negative curvature near infinity. Vasy \cite{vasy2013a,vasy2013b} had a new construction of the meromorphic continuation, which is the one used in \cite{Dyatlov_2016}.

The standard Patterson--Sullivan gap \cite{patterson1976limit,Sullivan1979} says
\begin{align}
    R(\lambda) \text{ has only finitely many poles in }\left\{\Im(\lambda)\geq -\max\left(0,\frac{d}{2}-\delta_\Gamma\right)\right\}.
\end{align}
Moreover, there is no pole in $\{\Im(\lambda) > \delta_\Gamma-d/2\}$ and there are conditions on $\delta_\Gamma$ such that $\lambda=i(\delta_\Gamma-d/2)$ is the first pole (see \cite{Sullivan1979,patterson1988lattice}). Using methods of \cite{Dyatlov_2016}, we can improve the essential spectral gap when $\delta_\Gamma\leq d/2$.

\begin{theorem}\label{scattering theory}
Let $M$ be a noncompact convex cocompact hyperbolic $d+1$-fold such that $\Gamma = \pi_1(M)$ is Zariski dense in $SO(d + 1, 1)_0$.
Let $\delta_\Gamma \in (0, d)$ be the Hausdorff dimension of the limit set $\Lambda(\Gamma)$.
Then there exists $\varepsilon_0 > 0$ such that for any $\varepsilon > 0$, $R(\lambda)$ has only finitely many poles $\lambda$ with $\Im \lambda > \delta_\Gamma - \frac{d}{2} - \varepsilon_0 + \varepsilon$. Moreover, for any $\chi\in C_0^\infty(M)$, there exists $C_0=C_0(\varepsilon)>0$ and $C=C(\varepsilon,\chi)>0$ such that
\begin{align}\label{resolvent estimate}
    \|\chi R(\lambda)\chi\|_{L^2\to L^2}\leq C|\lambda|^{-1-2\min(0,\Im\lambda)+\varepsilon},\quad |\lambda|>C_0,\quad \Im \lambda\in\left[\delta_\Gamma - \frac{d}{2} - \varepsilon_0 + \varepsilon,1\right].
\end{align}
\end{theorem}

In \cite[Theorem 2]{Dyatlov_2018}, Dyatlov--Jin showed Theorem \ref{scattering theory} with $d=1$ by proving Theorem \ref{main theorem} for $d=1$ and $X$ and $Y$ $\delta$-regular and applying \cite[Theorem 3]{Dyatlov_2016}; our result is the natural higher-dimensional generalization of this theorem.

The spectral gap in Theorem \ref{scattering theory} was first proved by Naud \cite{Naud2005} in dimension $2$ and generalized by Stoyanov \cite{stoyanov2011spectra} to higher dimensions. The size of their gap is implicit but our method gives an explicit constant $\varepsilon_0$ as in \eqref{epsilon0} depending on the fractal dimension $\delta_\Gamma$, the regularity constant and the nonorthogonality constant of the limit set $\Lambda(\Gamma)$. We give a method for computing nonorthogonality constants from the generators of a classical Schottky group $\Gamma\subset SL(2,\CC)$ in \S\ref{how to compute}.

Another advantage of the method of \cite{Dyatlov_2016} is that we also get the resolvent estimate (\ref{resolvent estimate}), which is hard to obtain using transfer operator techniques and in partiular is not included in \cite{Naud2005, stoyanov2011spectra}.

\begin{corollary}\label{Selberg zeta function}
Let $M$ be convex cocompact with $\Gamma$ Zariski dense.
Let $\zeta_M$ be the \dfn{Selberg zeta function}
\begin{align*}
    \zeta_M(s)=\prod\limits_{l\in\mathcal{L}_M}\prod\limits_{k=0}^\infty (1-e^{-(s+k)l}), \quad s=\frac{d}{2}-i\lambda
\end{align*}
where $\mathcal{L}_M$ consists of the lengths of all primitive closed geodesics on $M$ (with multiplicity). Then $\zeta_M(s)$ has only finitely many singularities (i.e. zeroes or poles) in the half plane $\{\Re s>\delta_\Gamma-\epsilon_0+\epsilon\}$ for any $\epsilon>0$.
\end{corollary}
\begin{proof}
This follows from Theorem \ref{scattering theory} and \cite{bunke1999group, patterson2001divisor}.
\end{proof}



The spectral gap is closely related to asymptotics of closed geodesics and exponential decay of correlations, which are important and well-studied questions in dynamical systems. We list a few references.
\begin{itemize}
    \item Chernov \cite{Chernov98} gave the first dynamical proof showing sub-exponential decay of correlations for $3$-dimensional contact Anosov flows.
    The groundbreaking work of Dolgopyat \cite{Dolgopyat_98} showed exponential decay of correlations for transitive Anosov flows with jointly nonintegrable $C^1$ stable/unstable foliations.
    \item Naud \cite{Naud2005} applied Dolgopyat's method to prove Theorem \ref{scattering theory} in dimension $2$.
    \item Stoyanov \cite{stoyanov2008spectra,stoyanov2011spectra} showed exponential mixing for a general class of Axiom A flows satisfying his local non-integrability condition.
    \item Sarkar--Winter \cite{sarkar2021exponential} used Dolgopyat's method to prove exponential mixing of the frame flow for convex cocompact hyperbolic manifolds. Chow--Sarkar \cite{chow2022exponential} extended it to locally symmetric spaces. 
\end{itemize}
All the above works require certain \dfn{nonintegrability conditions} which should be thought as the analogue of our nonorthogonality condition \eqref{nonorthogonality formula}.

We would like to mention some other related works on the spectral gap for convex cocompact hyperbolic manifolds.
\begin{itemize}
    \item Dyatlov--Zahl \cite{Dyatlov_2016}, Dyatlov--Jin \cite{Dyatlov_2018} and Bourgain--Dyatlov \cite{Bourgain_2018} proved the fractal uncertainty principle for $d=1$ and hence gave explicit essential spectral gaps.
    \item Bourgain--Dyatlov \cite{Bourgain_2017} used Fourier decay of the Patterson--Sullivan measure to get a spectral gap that only depends on $\delta_\Gamma$ when $d=1,\delta_\Gamma\leq 1/2$. This is generalized to Kleinian Schottky groups when $d=2$ by Li--Naud--Pan \cite{li2021kleinian} but in this case the spectral gap will depend on $\delta_\Gamma$ and another quantity related to our non-orthogonality constant $c_N$ (see \cite[Lemma 4.4]{li2021kleinian}). See also recent work of Khalil \cite{khalil2023exponential} for a method using additive combinatorics.
    \item Oh--Winter \cite{oh2016uniform} showed a uniform spectral gap for a large family of congruence arithmetic surfaces, which was then generalized to arbitrary dimensions by Sarkar \cite{Sarkar22}.
\end{itemize}

\subsection{Idea of the proof}
\subsubsection{Model problem: Arithmetic Cantor sets}\label{model problem}
We first describe the problem in the model case that $X, Y$ are arithmetic Cantor sets. Let $M \geq 3$ be an integer and $A, B \subseteq \{0,1, \dots, M-1\}^d$ be sets with
$\delta_A := \frac{\log|A|}{\log(M)}, \delta_B := \frac{\log|B|}{\log(M)} \le \frac{d}{2}$. We let $N := M^k$ and define the \dfn{arithmetic Cantor sets}
$$C_{k, A} := \{a_0 + a_1M + \cdots + a_{k-1} M^{k-1}: a_i \in A\}$$
$$C_{k, B} := \{b_0 + b_1M + \cdots + b_{k-1} M^{k-1}: b_i \in B\}.$$
We introduce the \dfn{discrete Fourier transform}
$$\mathcal F_N f(j) := N^{-d/2} \sum_{\ell \in \{0,1, \dots, N-1\}^d} \exp\left(2\pi ij \cdot \frac{\ell}{N}\right) f(\ell).$$
The fractal uncertainty principle states that there exists some $\varepsilon_0 > 0$ such that
\begin{equation}\label{DiscreteFUP}
\|1_{C_{k, A}} \mathcal{F}_N 1_{C_{k, B}}\|_{\ell^2 \to \ell^2} \lesssim N^{-\beta - \varepsilon_0}
\end{equation}
where $\beta := \frac{d - \delta_A - \delta_B}{2}$ \cite[\S3]{Dyatlov_2017}. Analyzing the Hilbert--Schmidt norm, we have
\begin{equation}\label{HSNorm}
\|1_{C_{k, A}} \mathcal{F}_N 1_{C_{k, B}}\|_{\ell^2 \to \ell^2} \le \|1_{C_{k, A}} \mathcal{F}_N 1_{C_{k, B}}\|_{HS} = \sqrt{\frac{|A|^k |B|^k}{N^d}} = N^{-\beta}.
\end{equation}
Thus, our goal is to obtain additional gain beyond $\beta$. To prove this, one can show as in \cite[Lemma 6.4]{Dyatlov_2019} that if we let
$$r_k := \|1_{C_{k, A}} \mathcal{F}_N 1_{C_{k, B}}\|_{\ell^2 \to \ell^2},$$
then $r_{k_1 + k_2} \le r_{k_1}r_{k_2}$. This can be used to show that if we can get any gain at all at some scale $k$, then we get a gain on all further levels, so we suppose for the sake of contradiction that we cannot obtain any gain at any scale, or that {the inequality present in (\ref{HSNorm}) is an equality. Then 
since the Hilbert--Schmidt norm measures the square root of the sum of squares of the singular values and the operator norm measures the largest singular value, it follows that the operator $N^{d/2} 1_{C_{k, A}} \mathcal{F}_N 1_{C_{k, B}}$ must be rank one. A simple computation then shows that the operator $N^{d/2} 1_{C_{k, A}} \mathcal{F}_N 1_{C_{k, B}}$ is the matrix $(\exp(2\pi i j\cdot \ell/N))_{j \in C_{k, A}, \ell \in C_{k, B}}$ (and is zero in the unspecified entries). Computing the determinant of $2 \times 2$ minors, we see that
\[ \left|\det\begin{pmatrix} \exp(2\pi i j \cdot \ell/N) & \exp(2\pi i j' \cdot \ell/N) \\ \exp(2\pi i j \cdot \ell'/N) & \exp(2\pi i j' \cdot \ell'/N) \end{pmatrix}\right| = \left|\exp\left(2\pi i \frac{\langle j - j', \ell - \ell' \rangle}{N}\right) - 1\right| = 0\]
for all $j, j' \in C_{k, A}$ and $\ell, \ell' \in C_{k, B}$. Thus, (\ref{DiscreteFUP}) holds as long as a \emph{nonorthogonality} condition
$$\langle j - j', \ell - \ell' \rangle \neq 0$$
holds for some choice of $j, j' \in A, \ell, \ell' \in B$. 
If non-orthogonality is violated at all scales, then \eqref{DiscreteFUP} cannot hold, see Example \ref{nonorthogonality is necessary}.

\subsubsection{Nonorthogonality and Dolgopyat's method}
Our proof and the proof of \cite{Dyatlov_2018} lies in the continuous setting where the fractal is not necessarily self-similar. Thus, we must construct a tree of tiles that discretizes the doubling measure $\mu$, and which is regular enough so that each tile has two children which are spaced far apart away. While very nice submultiplicativity does not hold as it does in the discrete case, we can still, via an induction on scales argument, propagate gain on one scale to gain on all scales. The key tool allowing us to obtain gain on all scales is nonorthogonality, which we formulated in (\ref{nonorthogonality formula}); it asserts that we can find many points in the intersections of the vertical and horizontal laminations where the phase is ``oscillating faster than the function $\mathcal B_h$ is being tested against" at every scale, and so we must obtain a gain at every scale. This technique, called \dfn{Dolgopyat's method}, has been used to obtain fractal uncertainty principles, spectral gaps, or exponential mixing in previous works, including \cite{Dolgopyat_98, Naud2005, stoyanov2008spectra, stoyanov2011spectra, Dyatlov_2018, Tsujii2023}.

The improvement on each child is measured in the spaces $C_\theta(I)$ that were introduced in \cite[Lemma 5.4]{Naud2005}. Informally speaking, localizations of $\mathcal{B}_h$ to a tile $I$ have roughly constant oscillation when normalized by $\theta \diam(I)$ for some appropriate choice of $\theta$ \cite[\S2.2]{Dyatlov_2018}. The $C_\theta(I)$ norms are meant to capture this fact and to measure cancellation on scale $I$, similar to how algebraic manipulations on $M^k$-dimensional vectors can be used to measure cancellation in the arithmetic Cantor case.

\subsubsection{Improvements over Dyatlov--Jin}
The method of Dyatlov--Jin \cite{Dyatlov_2018} does not immediately generalize to $d \geq 2$, for two reasons.
First, in order to ensure that each interval has at least two children that are sufficiently far apart, Dyatlov--Jin allow intervals of varying length to appear in the tree by merging together consecutive intervals that intersect the fractal. However, in higher dimensions this leads to long, narrow, winding tiles appearing in the tree; these do not satisfy suitable doubling estimates, as exemplified by the following example.

\begin{example}
Let $X$ be a Sierpi\'nski carpet, and consider the merged discretization for $X$ (see \S\ref{discretization} or \cite[\S2.1]{Dyatlov_2018}).
Since $X$ is path-connected, every scale consists of a single tile, the only child of the single tile at the previous scale! It is impossible to prove that every tile has two children which enjoy phase cancellation.

However, our method must be able to handle the Sierpi\'nski carpet, since it meets the hypotheses of Corollary \ref{FUP classic} if it is embedded in $\RR^4$. Indeed, $2\delta_X \approx 3.8 < 4$. Moreover, $X$ is nonorthogonal to itself at one scale (see the figure), so it is at every scale by self-similarity.

\begin{figure}[h]

\centering
\includegraphics[width=0.5\textwidth]{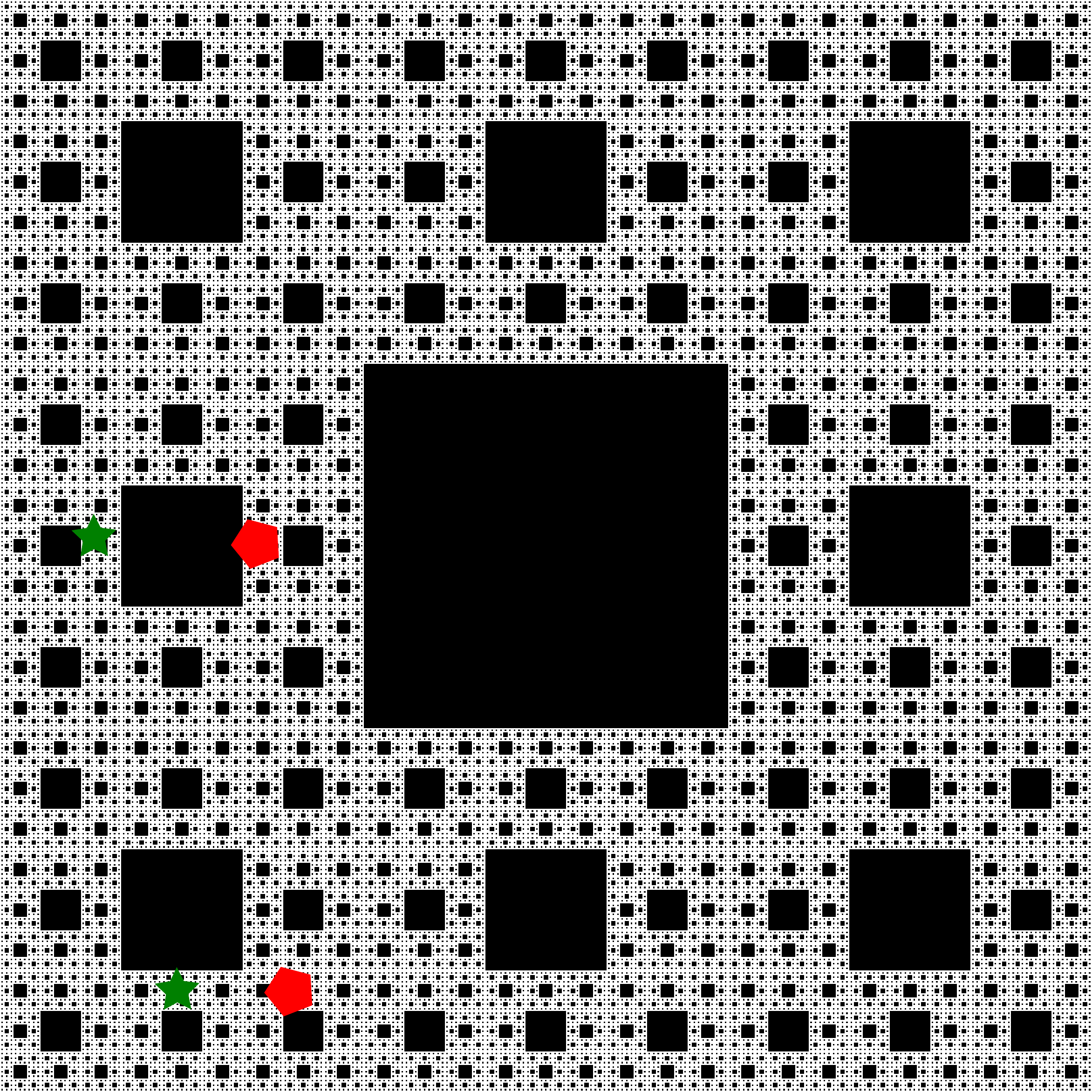}
\caption{Nonorthogonality of the Sierpi\'nski carpet $X$ (the white region) to itself at scale $\frac{1}{3}$ (where $\diam X = \sqrt 2$). Given any two green points $x_1, y_1 \in X$, we can find two red points $x_2, y_2 \in X$ such that $|x_1 - x_2|$ and $|y_1 - y_2|$ are both $\approx 0.15$, and $|\sin \angle(x_2 - x_1, y_2 - y_1)| \ll 1$, so $(X, X)$ is nonorthogonal with constant $(3 \cdot 0.14)^2 \approx 0.42$. Adapted from \cite{Rossel08}.}
\end{figure}
\end{example}

Secondly, as remarked above, one cannot obtain cancellation for arbitrary children $I_1, I_2$, but only those which are ``not orthogonal to each other". Otherwise, even if we construct $I_1, I_2$ to be the appropriate distance each other to impose cancellation, it will not follow that the phases actually cancel each other.

\begin{example}\label{nonorthogonality is necessary}
Let $X := [-5, 5] \times \{0\}$ and $Y := \{0\} \times [-5, 5]$. The Gaussian
$$f(x, y) := e^{-\frac{x^2}{2} - \frac{y^2}{2h^2}}$$
is localized to $X_{5h}$ and its Fourier transform is localized to $Y_{5h}$. So the fractal uncertainty principle is simply false for $(X, Y)$, even though $\delta_X + \delta_Y = 2 \leq 2$, and we must use the nonorthogonality hypothesis somehow. One can also see if $X'\subset X$ and $Y'\subset Y$ are fractals, then fractal uncertainty principle does not hold for $(X',Y')$.
\end{example}

To overcome these difficulties, we improve on Dyatlov--Jin as follows:
\begin{enumerate}
\item We carefully construct the tree, so that tiles in the tree are very close to cubes, and therefore satisfy good doubling estimates, but also so that each tile contains two children a suitable distance from each other. 
\item We prove that if $X, Y$ are nonorthogonal, then tangent vectors to $X, Y$ satisfy a \dfn{reverse Cauchy--Schwarz inequality} which ensures that the phases cannot decouple.
\end{enumerate}
These goals are accomplished by Proposition \ref{discret}, which asserts that we can construct the so-called \dfn{perturbed standard discretization} of $\mu$, and 
Proposition \ref{random discretization}, which asserts that many quadruples of tiles in the perturbed standard discretization satisfy the desired spacing and reverse Cauchy--Schwarz inequality.

We found it convenient to use the language of probability theory to state Proposition \ref{random discretization}, as we then could interpret the various quantities appearing in the induction on scale (Proposition \ref{statistical iteration}) as expected values or variances of certain averages of $\mathcal B_h f$ taken over random tiles. The necessary estimates needed to obtain a contradiction then follow from the \dfn{second moment method} -- namely, the observation that, if Proposition \ref{statistical iteration} is false, then the variance of such random variables is impossibly small given the large size of their tails. A similar approach was taken by \cite{Dyatlov_2018}, which used the strict convexity of balls in Hilbert spaces \cite[Lemma 2.7]{Dyatlov_2018} to accomplish the same goals.

\subsection{Outline of the paper}
In \S\ref{preliminaries} we recall some preliminaries.

In \S\ref{discretization} we construct our discretization and show that it has good statistical properties, as made precise by Proposition \ref{random discretization}.

In \S\ref{induction on scale} we carry out our inductive argument. The main proposition is the iterative step, Proposition \ref{statistical iteration}; we then use this to prove Theorem \ref{main theorem}.

We then turn to the applications in \S\ref{applications} where we reduce Corollary \ref{FUP classic} and Theorem \ref{scattering theory} to Theorem \ref{main theorem} by standard techniques.

\subsection{Acknowledgments}
The authors would like to thank Semyon Dyatlov for suggesting this problem and for helpful comments on earlier drafts.
We would also like to thank Frédéric Naud for suggesting the references \cite{stoyanov2008spectra, stoyanov2011spectra, sarkar2021exponential}, Pratyush Sarkar for suggesting the references \cite{Sarkar22, Mazzeo87, guillope1995polynomial}, Terence Tao for helpful discussions and for suggesting the reference \cite{christ1990b}, Qiuyu Ren for proposing the method we use in \S\ref{how to compute}, and Long Jin and Ruixiang Zhang for helpful discussions.

AB was supported by the National Science Foundation's Graduate Research Fellowship Program under Grant No. DGE-2040433, JL was supported by the NSF's GRFP under Grant No. DGE-2034835, and ZT was partially supported by the NSF grant DMS-1952939 and Simons Targeted Grant Award No. 896630.


\section{Preliminaries}\label{preliminaries}
\subsection{Probability theory}
We shall have probability spaces $A, B$, and will denote by $a, a', a''$ and $b, b', b''$ outcomes in those spaces (or equivalently random variables with values in $A, B$). The expected value of a random variable $X$ is denoted $\Expect X$, while $\Expect(X|E)$ refers to the conditional expectation of $X$ assuming an event $E$. The probability of the event $E$ is denoted $\Pr(E)$, and the variance of a random variable is 
$$\Var X := \Expect(X^2) - (\Expect X)^2.$$
If $X, Y$ are i.i.d., then 
$$\Expect |X - Y|^2 = \Expect |X|^2 + \Expect |Y|^2 - 2\Expect(XY) = 2(\Expect |X|^2 - (\Expect X)^2)$$
and so 
\begin{equation}\label{variance in a product space}
\Expect |X - Y|^2 = 2 \Var X.
\end{equation}
We also record \dfn{Cantelli's inequality}, valid for any constant $\lambda > 0$ \cite[Theorem 1]{Lugosi09}:
\begin{equation}\label{Cantelli}
\Pr(X \geq \Expect X + \lambda) \leq \frac{\Var X}{\lambda^2 + \Var X}.
\end{equation}

\subsection{A geometric mean value theorem}
We shall need an analogue of the mean value theorem for phase functions \cite[Lemma 2.5]{Dyatlov_2018}. To formulate it, we shall recall some differential geometry.

If $R$ is a nondegenerate rectangle in $\RR^d_x \times \RR^d_y$, and $v, w$ are unit tangent to the edges of $R$, then we write $\gamma_R := v \otimes w$ for the \dfn{unit bitangent} to $R$\footnote{Strictly speaking, the unit bitangent should be defined using the exterior algebra, but since $R$ is assumed nondegenerate this adds more complication for no gain.} and $\dif A_R$ for the area element on $R$.
We will consider the case that $v \in \RR^d_x$ and $w \in \RR^d_y$.
In that case, $\gamma_R$ and the off-diagonal Hessian $\partial^2_{xy} \Phi$ both lie in $\RR^d_x \otimes \RR^d_y$, so we can consider their contraction
$$\langle \partial^2_{xy} \Phi, \gamma_R\rangle = \partial_v \partial_w \Phi.$$

\begin{lemma}
Let $\Phi \in C^2(\RR^d \times \RR^d)$. Let $x_0, x_1, y_0, y_1 \in \RR^d$, and let $R$ be the rectangle with vertices $(x_i, y_j)$, $i,j \in \{0, 1\}$. Then 
\begin{equation}\label{geometric MVT}
\int_R \langle \partial^2_{xy} \Phi, \gamma_R\rangle \dif A_R = \Phi(x_0, y_0) - \Phi(x_0, y_1) - \Phi(x_1, y_0) + \Phi(x_1, y_1).
\end{equation}
\end{lemma}
\begin{proof}
Both sides of (\ref{geometric MVT}) are preserved by orientation-preserving isometries which preserve the product structure on $\RR^d \times \RR^d$. 
In particular, we may take $x_0, y_0 = 0$, $x_1 = (\xi^*, 0, \dots, 0)$, and $y_1 = (\eta^*, 0, \dots, 0)$ for some $\xi^*, \eta^* \in \RR$.
We then set
$$\varphi(\xi, \eta) := \Phi((\xi, 0, \dots, 0), (\eta, 0, \dots, 0)).$$
Then by Fubini's theorem,
\begin{align*}
    \int_R \langle \partial^2_{xy} \Phi, \gamma_R\rangle \dif A_R
    &= \int_0^{\xi^*} \int_0^{\eta^*} \partial_\xi \partial_\eta \varphi(\xi, \eta) \dif \eta \dif \xi \\
    &= \int_0^{\xi^*} \partial_\xi \varphi(\xi, \eta^*) - \partial_\xi \varphi(\xi, 0) \dif \xi \\
    &= \varphi(\xi^*, \eta^*) - \varphi(\xi^*, 0) - (\varphi(0, \eta^*) - \varphi(0, 0)) \\
    &= \Phi(x_0, y_0) - \Phi(x_0, y_1) - \Phi(x_1, y_0) + \Phi(x_1, y_1). \qedhere
\end{align*}
\end{proof}

We now estimate the difference between (\ref{geometric MVT}) evaluated over two different rectangles $R, R'$ by differentiating $\Phi$ along a homotopy between $R, R'$. This estimate will be useful when applying the nonorthogonality hypothesis.

\begin{lemma}
Let $\Phi \in C^3(\RR^d \times \RR^d)$ and let $R_t = [x_0(t), x_1(t)] \times [y_0(t), y_1(t)]$, where $t = 0, 1$ and $x_i(t), y_i(t) \in \RR^d$.
Let $\gamma_t := \gamma_{R_t}$ be the unit bitangent to $R_t$.
Assume that for some $0 \leq \varepsilon_x, \varepsilon_y, c_x, c_y \leq 1$:
\begin{enumerate}
    \item $|x_i(1) - x_i(0)| \leq \varepsilon_x$ and $|y_i(1) - y_i(0)| \leq \varepsilon_y$.
    \item $|x_1(t) - x_0(t)| \leq c_x$ and $|y_1(t) - y_0(t)| \leq c_y$.
\end{enumerate}
Then
\begin{equation}\label{difference of MVTs}
    \left|\int_{R_1} \langle \partial^2_{xy} \Phi, \gamma_1\rangle \dif A_{R_1} - \int_{R_0}  \langle \partial^2_{xy} \Phi, \gamma_0 \rangle \dif A_{R_0}\right| \leq 7 \|\partial^2_{xy} \Phi\|_{C^1} (\varepsilon_x c_y + \varepsilon_y c_x).
\end{equation}
\end{lemma}
\begin{proof}
By taking convex combinations, we define $x_i(t)$ and $y_i(t)$ for any $t \in [0, 1]$, hence also $R_t$ and $\gamma_t$.
Now introduce the parametrization 
$$\Psi_t(\xi, \eta) :=
\begin{bmatrix} \xi x_1(t) + (1 - \xi) x_0(t) \\ \eta y_1(t) + (1 - \eta) y_0(t)
\end{bmatrix} \in \RR^d \times \RR^d$$
which maps $[0, 1]^2$ to $R_t$.
Also let $v_t := x_1(t) - x_0(t)$ and $w_t := y_1(t) - y_0(t)$, so $|v_t| |w_t|$ is the (unoriented) Jacobian of the map $\Psi_t$.
We record for later that $|v_t| \leq c_x$ and $|w_t| \leq c_y$.

We estimate 
\begin{align*}
&\left|\int_{R_1} \langle \partial^2_{xy} \Phi, \gamma_1\rangle \dif A_{R_1} - \int_{R_0}  \langle \partial^2_{xy} \Phi, \gamma_0 \rangle \dif A_{R_0}\right| \\
&\qquad = \left|\int_0^1 \partial_t \int_{R_t} \langle \partial^2_{xy} \Phi, \gamma_t\rangle \dif A_{R_t} \dif t\right|  \\
&\qquad \leq \int_0^1 \int_0^1 \int_0^1 |\partial_t (\langle \partial^2_{xy} \Phi \circ \Psi_t(\xi, \eta), \gamma_t\rangle \cdot |v_t| \cdot |w_t|)| \dif \xi \dif \eta \dif t.
\end{align*}
We next split up 
\begin{align*}
&|\partial_t (\langle \partial^2_{xy} \Phi \circ \Psi_t(\xi, \eta), \gamma_t\rangle |v_t| |w_t|)| \\
&\qquad \leq |\langle \partial_t(\partial^2_{xy} \Phi \circ \Psi_t(\xi, \eta)), \gamma_t\rangle| \cdot |v_t| \cdot |w_t|
+ |\langle \partial^2_{xy} \Phi \circ \Psi_t(\xi, \eta), \partial_t \gamma_t\rangle| \cdot |v_t| \cdot |w_t| \\
&\qquad \qquad + |\langle \partial^2_{xy} \Phi \circ \Psi_t(\xi, \eta), \gamma_t\rangle| \cdot |\partial_t |v_t|| \cdot |w_t|
+ |\langle \partial^2_{xy} \Phi \circ \Psi_t(\xi, \eta), \gamma_t\rangle| \cdot |v_t| \cdot |\partial_t |w_t|| \\
&\qquad =: \mathbf{I} + \mathbf{II} + \mathbf{III} + \mathbf{IV}.
\end{align*}
To estimate $\mathbf I$ we compute
$$\partial_t \Psi_t(\xi, \eta) = \begin{bmatrix}\xi (x_1(1) - x_1(0)) + (1 - \xi)(x_0(1) - x_0(0)) \\ \eta (y_1(1) - y_1(0)) + (1 - \eta)(y_0(1) - y_0(0))
\end{bmatrix}$$
and conclude that $\|\partial_t \Psi_t\|_{C^0} \leq \varepsilon_x + \varepsilon_y$.
Therefore, by the chain rule, 
$$\mathbf I \leq \|\nabla \partial^2_{xy} \Phi\|_{C^0} \|\partial_t \Psi_t\|_{C^0} |v_t| \cdot |w_t| \leq \|\partial^2_{xy} \Phi\|_{C^1} c_x c_y (\varepsilon_x + \varepsilon_y) \leq \|\partial^2_{xy} \Phi\|_{C^1} (c_x \varepsilon_y + c_y \varepsilon_x).$$
We furthermore estimate 
$$|\partial_t v_t| = |x_1(1) - x_1(0) - x_0(1) + x_0(0)| \leq 2\varepsilon_x$$
and similarly for $w_t$.
Now to estimate $\mathbf{II}$, we recall
$$\gamma_t = \frac{v_t}{|v_t|} \otimes \frac{w_t}{|w_t|}.$$
So by the product rule,
\begin{align*}
    |\partial_t \gamma_t|
    &\leq \frac{2}{|v_t|} |\partial_t v_t| + \frac{2}{|w_t|} |\partial_t w_t|
    \leq 4\left[\frac{\varepsilon_x}{|v_t|} + \frac{\varepsilon_y}{|w_t|}\right].
\end{align*}
So 
$$\mathbf{II} \leq 4 \|\partial^2_{xy} \Phi\|_{C^0} (c_x \varepsilon_y + c_y \varepsilon_x) \leq 4 \|\partial^2_{xy} \Phi\|_{C^1} (c_x \varepsilon_y + c_y \varepsilon_x).$$
To estimate $\mathbf{III}$, we use Kato's inequality $|\partial_t |v_t|| \leq |\partial_t v_t|$ to bound
\begin{align*}
    \mathbf{III} \leq 2\|\partial^2_{xy} \Phi\|_{C^0} c_y \varepsilon_x \leq 2\|\partial^2_{xy} \Phi\|_{C^1} c_y \varepsilon_x.
\end{align*}
The estimate on $\mathbf{IV}$ is similar but with $x$ and $y$ swapped.
Adding up these terms and integrating, we conclude the result.
\end{proof}


\section{Discretization of sets and measures}\label{discretization}
\subsection{A new discretization}
As in previous works on the fractal uncertainty principle, such as \cite{Bourgain_2018, Dyatlov_2018}, we will discretize fractals as trees.

\begin{defi}
Let $X \subseteq \RR^d$ be a set. A \dfn{discretization} of $X$ is a family $V(X) = (V_n(X))_{n \in \ZZ}$ of sets, where $V_n(X)$ is a set of nonempty subsets of $\RR^d$ such that
\begin{itemize}
    \item $X=\bigcup\{I \cap X:I\in V_n(X)\}$ for each $n$ and the union is disjoint;
    \item for any $I\in V_n(x)$, there exist $I_k\in V_{n+1}(X)$ such that $I=\bigcup\limits_{k} I_k$.
\end{itemize}
Given $I\in \cup_n V_n(X)$, the \dfn{height} of $I$ is defined as $H(I)=\sup\{n: I\in V_n(X)\}$.
\end{defi}

\begin{defi}
For a compact set $X\subset\RR^d$ and base $L\geq 2$, its \dfn{standard $L$-adic discretization} $V^0 = (V_n^0)_{n\in\ZZ}$ is defined by: $I\in V_n^0(X)$ if and only if
$$I=I_n(q) := [q_1, L^{-n} + q_1) \times [q_2, L^{-n} + q_2) \times \cdots \times [q_d, L^{-n} + q_d)$$
for some $q \in L^{-n} \ZZ^d$ and $I\cap X\neq \varnothing$. 
\end{defi}

The standard discretization was used by Bourgain--Dyatlov \cite{Bourgain_2018} to prove the fractal uncertainty principle in the case $d = 1$, $\delta > 1/2$.
The problem with the standard discretization is that a box in $V_n^0(X)$ may be too small for the fractal measure.
Dyatlov--Jin \cite{Dyatlov_2018} addressed this issue in the case $d = 1$, $\delta \leq 1/2$, by considering a discretization that we call the \dfn{merged discretization}.
Unfortunately, if $d \geq 2$ and $\delta \geq 1$, then the merged discretization does not satisfy the desirable estimates, as intimated by the fact that such estimates have a constant of the form $O(1)^{\frac{1}{\delta(1 - \delta)}}$ for $\delta < 1$ in \cite{Dyatlov_2018}. \\\\
We now construct a discretization which is more adapted to our setting. Given a compact convex set $I$ and a real number $\alpha > 0$, we denote by $I\alpha$ the dilation of $I$ by $\alpha$ from its barycenter. For $A,B\subset \RR^d$, we use the $\ell^\infty$ Hausdorff distance
\begin{align*}
    \dist(A,B)=\sup\{|a_i-b_i|:1\leq i\leq d, a=(a_i)\in A, b=(b_i)\in B\}.
\end{align*}

\begin{prop}\label{discret}
For a compact set $X\in\RR^d$, $N\in\NN$, $L \geq 10^3$,
there is a discretization $V(X)$ of $X$ such that for $I\in V_n(X)$, $1\leq n\leq N$,
\begin{itemize}
    \item there exists $I^0\in V^0_n(X)$ such that
\begin{equation}\label{perturbed dist is perturbation}
I^0(1-L^{-2/3}) \subset I\subset I^0(1+L^{-2/3}),
\end{equation}
    \item and there exists a point $x_0$ in $X \cap I$ such that
\begin{equation}\label{good cube condition}
    \dist(x_0, \partial{I}) \geq \frac{1}{10}L^{-2/3-n}.
\end{equation}
\end{itemize}
\end{prop}

We call this discretization the \dfn{perturbed standard discretization}, and we call elements of the perturbed standard discretization \dfn{tiles} (to emphasize that they may not be cubes).

\begin{remark}
Christ \cite{christ1990b} constructed dyadic cubes with similar properties for metric spaces with a doubling measure $\mu$ as in Definition \ref{def:doubling measure}. It's possible that the construction there can also be applied to prove Theorem \ref{main theorem}. Our construction is less general but does not rely on the existence of a doubling measure.
\end{remark}

\begin{figure}[h]
\centering
\includegraphics[width=0.30\textwidth]{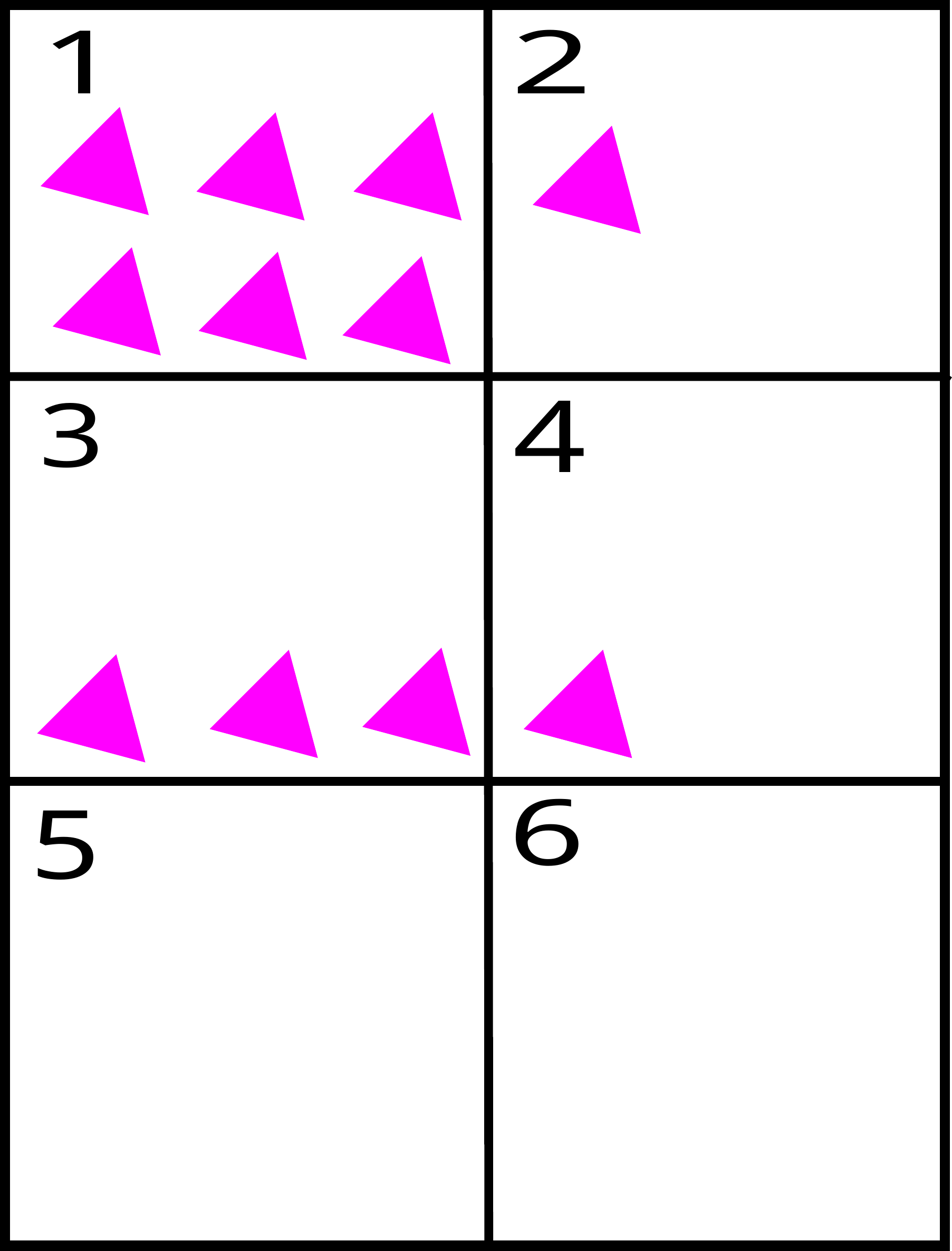}
\includegraphics[width=0.30\textwidth]{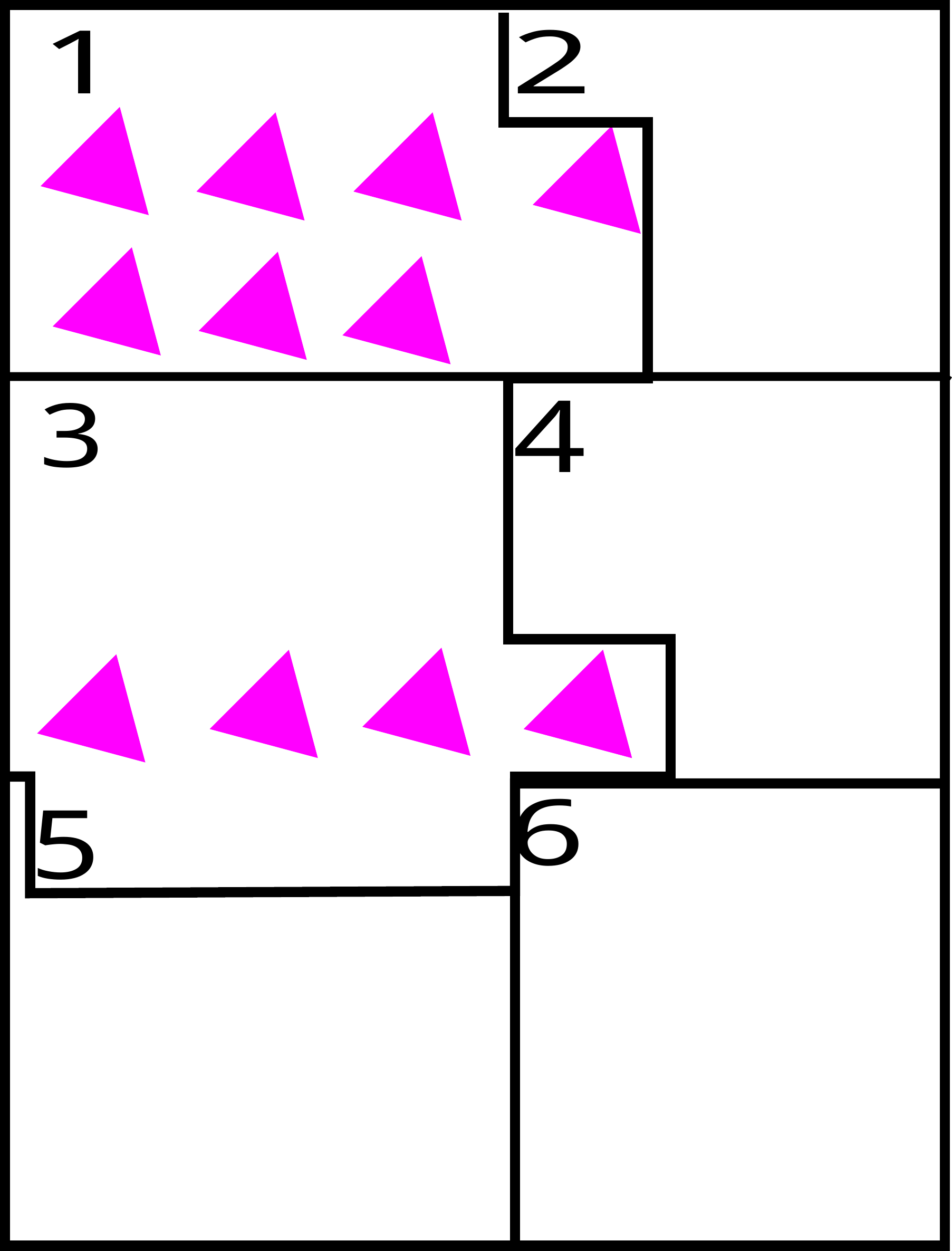}
\caption{A standard (left) and perturbed standard (right) discretization. On the left, Cube 1 is type $2$, Cubes 2 and 3 are type $1$, Cube 4 is type $0$, and Cubes 5 and 6 are type $-1$; on the right, Tiles 1 and 3 are good and all other tiles are type $-1$.}
\label{figure2}
\end{figure}

\subsection{Constructing the new discretization}
\subsubsection{Preliminaries}
We establish some terminology and notation that we will use in the construction of the new discretization.
Let $I$ be a cube, such that $\bar{I}=[a_1,b_1]\times \cdots [a_d,b_d]$. For $1\leq k\leq d$, define the $k$-boundary
\begin{align*}
    \partial^k I:=\bigcup\limits_{j_1,\cdots,j_k} [a_1,b_1]\times \cdots \times \{a_{j_i},b_{j_i}\}\times \cdots\times [a_d,b_d].
\end{align*}
For a set $A\in \RR^d$, $r>0$, let the $\ell^\infty$ ball around $A$ with radius $r$ be 
\begin{align*}
    B_\infty(A,r)=\{x\in\RR^d:\exists a\in A, |a-x|_{\ell^\infty}<r\}.
\end{align*}
We stress that a $B$ without a subscript refers to the $\ell^2$ ball (and in particular, the balls in the definition of nonorthogonality are $\ell^2$ balls!)

For a subset $P\subset \partial^kI$ of the $k$-boundary of a cube $I$, suppose without loss of generality that 
\begin{align*}
    P\subseteq \{a_1,b_1\}\times \cdots \times \{a_k,b_k\}\times [a_{k+1},b_{k+1}]\times \cdots\times [a_d,b_d].
\end{align*}
In that case, we define the \dfn{tubular neighbourhoods}
\begin{align*}
    B^t_\infty(P,r):=\{x\in\RR^d: \exists y=(y_i)\in P, |x_1-y_1|<r,\cdots, |x_{k}-y_{k}|<r, x_{k+1}=y_{k+1},\cdots, x_d=y_d \}
\end{align*}
and
\begin{align*}
    B^t_\infty(P,r_1,r_2):=B^t_\infty(P,r_1)\cup B_\infty(P,r_2).
\end{align*}
We define other cases similarly.

Let $V^0(X)$ be the standard discretization and $n \leq N$. We divide the cubes $I \in V_n^0(X)$ into the following types:
\begin{itemize}
    \item $I$ is of type $d$ if there exists a point $x\in X\cap I$ such that $\dist(x,\partial I)>L^{-2/3-n}/2$;
    \item $I$ is of type $d-1$ if there exists a point $x\in X\cap I$ with $\dist(x,\partial^2 I)>L^{-2/3-n}/2$ but it is not of type $d$;
    \item $I$ is of type $d-2$ if there exists a point $x\in X\cap I$ with $\dist(x,\partial^3 I)>L^{-2/3-n}/2$, but not of type $\geq d - 1$;
    \item $\cdots$;
    \item $I$ is of type $0$ if $X\cap I$ is nonempty and $\dist(X\cap I,\partial^{d}I)\leq L^{-2/3-n}/2$;
    \item $I$ is of type $-1$ if $X \cap I$ is empty.
\end{itemize}
See Figure \ref{figure2}.

We want to modify the cubes $I \in V_n^0(X)$ into tiles $T$ so that there exists $x_0 \in X \cap T$ satisfying
\begin{equation}\label{good cube condition 2}
    \dist(x_0, \partial{T}) \geq \frac{1}{5} L^{-2/3-n}.
\end{equation}
We say that a tile $T$ is \emph{good} if (\ref{good cube condition 2}) holds, and otherwise that it is \emph{bad}. For the remainder of the proof, we assume:

\begin{invariant}\label{badness invariant}
If a tile $T$ constructed from a cube $I$ is bad, then $T \subseteq I$.
\end{invariant}

This invariant is true at the current stage of the proof; we necessarily have $T = I$, since we have not modified any tiles yet.

We want to do induction on the type of the tiles. In order to do so, we will need a notion of ``type" for a bad tile. 
By Invariant \ref{badness invariant}, in order for type to be well-defined, it suffices to define the type of a tile $T$ which was modified from a cube $I$ such that $T \subseteq I$.
In that case, we define the \emph{type} of $T$ to be $k$ if $I$ is of type $k$ with respect to $X\cap T$; that is, if $I$ has type $k$ in $V_n(X \cap T)$ where $V(X \cap T)$ consists of the restriction of elements of $V(X)$ that we are already defined to $T$.

\subsubsection{Induction on type}
We now induct backwards on the largest type $k$ of a bad tile.
We make the following inductive assumptions, which are vacuous  at the start of the inductive process, when $k = d - 1$:

\begin{invariant}\label{bad tiles have low type}
Every bad tile has type $\leq k$.
\end{invariant}

\begin{invariant}\label{closeness invariant}
If a tile $T$ was constructed from a cube $I$, then $\dist(\partial T, \partial I) \leq L^{-n-2/3}/2$.
\end{invariant}

\begin{lemma}\label{improving invariants}
Assume that $0 \leq k \leq d - 1$, and the above set of tiles satisfies Invariants \ref{badness invariant}, \ref{bad tiles have low type}, and \ref{closeness invariant}. Then we may modify each tile to obtain a new set of tiles satisfying Invariants \ref{badness invariant}, \ref{bad tiles have low type}, and \ref{closeness invariant}, but with $k$ replaced by $k - 1$.
\end{lemma}

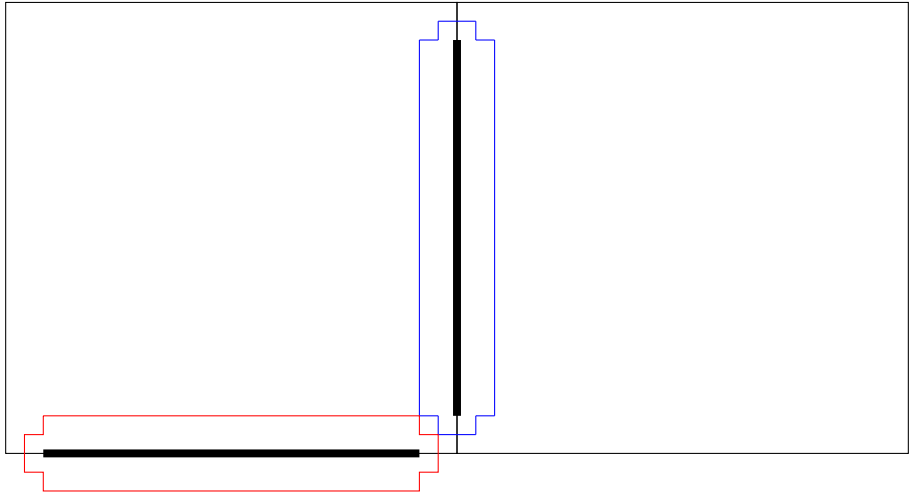
\begin{figure}
    \centering
\begin{tikzpicture}
\draw (-3,0) rectangle +(6,6) ;
\draw (3,0) rectangle +(6,6);
\draw [line width=3]  (3,5.5) -- (3,0.5) ;
\draw[blue] (2.75,5.75) -- (3.25,5.75) ;
\draw[blue]  (2.75,5.75) -- (2.75,5.5) ;
\draw[blue]  (2.5,5.5) -- (2.75,5.5) ;
\draw[blue]  (2.5,5.5) -- (2.5,0.5) ;
\draw[blue]  (2.5,0.5) -- (2.75,0.5) ;
\draw[blue]  (2.75,0.5) -- (2.75,0.25) ;
\draw[blue]  (2.75,0.25) -- (3.25,0.25) ;
\draw[blue]  (3.25,0.25) -- (3.25,0.5) ;
\draw[blue]  (3.25,0.5) -- (3.5,0.5) ;
\draw[blue]  (3.5,0.5) -- (3.5,5.5) ;
\draw[blue]  (3.5,5.5) -- (3.25,5.5) ;
\draw[blue]  (3.25,5.5) -- (3.25,5.75);
\draw[line width=3] (-2.5,0) -- (2.5,0);
\draw[red] (-2.5,0.5) -- (2.5,0.5) -- (2.5,0.25) -- (2.75,0.25) -- (2.75,-0.25) -- (2.5,-0.25) -- (2.5,-0.5) -- (-2.5,-0.5) -- (-2.5,-0.25) -- (-2.75,-0.25 ) -- (-2.75,0.25) -- (-2.5,0.25) -- (-2.5,0.5) ;
\end{tikzpicture}
\caption{The proof of Lemma \ref{improving invariants}, Case (2). The bold lines represent components $P$ of the boundary. The tubular neighborhoods around them do not intersect.\label{tubularfig}}
\end{figure}

\begin{proof}
Let $T$ be a bad tile of type $k$ modified from some cube $I$, and let $P$ be a connected component of $\partial^{k}I\setminus B_\infty(\partial^{k+1}I, L^{-2/3-n}/2)$ such that $B^t(P,L^{-n-2/3}/5)\cap X\cap T\neq \varnothing$.
We modify the adjacent tiles to $P$:
\begin{enumerate}
    \item If there is a good tile $T' \neq T$ adjacent to $P$, then we enlarge $T'$ to contain the tubular neighborhood $\mathcal T := B^t_\infty(P,L^{-2/3-n}/2) \cap (T' \cup T)$. Then:
\begin{enumerate}
    \item $T'$ is still good.
    \item $T$ no longer contains $P$.
    \item Since $\mathcal T$ is contained in $T' \cup T$, no other tile is affected.
\end{enumerate}
    \item Otherwise, by Invariant \ref{bad tiles have low type}, every tile adjacent to $P$ has type $\leq k$.
    In this case, we enlarge $T$ by a tubular neighborhood $\mathcal T := B^t_\infty(P,L^{-2/3-n}/2, L^{-2/3-n}/4)$. Then:
\begin{enumerate}
    \item $\mathcal T$ is disjoint from all other tubular neighborhoods of this form. See Figure \ref{tubularfig}.
    \item Prior to this step, every tile $T'$ adjacent to $P$ was bad, so by Invariant \ref{badness invariant}, $T'$ was contained in the cube $I'$ it was modified from. If $\mathcal T'$ is a tubular neighborhood transferred between tiles in a previous step, and $\mathcal T \cap \mathcal T'$ is nonempty, then there exists $T'$ adjacent to $P$ containing $\mathcal T'$, but $T'$ is not contained in $I'$, which is a contradiction. Therefore $\mathcal T$ is disjoint from all tubular neighborhoods transferred between tiles in a previous step.
    \item $T$ becomes good.
    \item Every tile $T' \neq T$ adjacent to $P$ no longer contains $P$.
\end{enumerate}
\end{enumerate}
We iterate the above procedure over all possible components $P$, stopping once there are no more components to consider. This happens after finitely many stages, because of the following facts:
\begin{enumerate}
\item If a tubular neighborhood of a component $P$ is absorbed by a tile $T$ of type $k$, and its other neighboring tile is $T'$, then $T$ becomes good, and $P$ can no longer witness that $T'$ has type $\geq k$. Therefore we will not iterate over $P$ again.
\item At each stage, no new bad tiles are created, and no bad tiles are given more points and remain bad. Therefore Invariants \ref{badness invariant} and \ref{bad tiles have low type} are preserved.
\item Invariant \ref{closeness invariant} is preserved, because if $T$ was constructed from $I$, then we only modify $T$ in a neighborhood of distance $L^{-n-2/3}/2$ of $\partial I$.
\end{enumerate}

After iterating over all possible components $P$, Invariant \ref{bad tiles have low type} is improved, so that every bad tile has type $\leq k - 1$.
Indeed, if $T$ is still bad, and was type $k$, then every tubular neighborhood of a component $P$ which could witness that $T$ had type $k$ was absorbed into a neighboring tile, so $T$ must have type $\leq k - 1$.
\end{proof}

After stage $k = 0$, every bad tile has type $-1$ by Invariant \ref{bad tiles have low type}.
However, if $T$ is a tile of type $-1$, then by definition $X \cap T \cap I$ is empty.
Then, by Invariant \ref{badness invariant}, $X \cap T$ is empty, and we may discard the tile $T$ entirely.

Let $\tilde V_n(X)$ be the set of good tiles that were constructed from $V_n(X)$ by the above procedure.
Then every tile in $\tilde V_n(X)$ satisfies (\ref{good cube condition 2}), and
$$X = \bigsqcup_{T \in \tilde V_n(X)} T \cap X.$$
However, $\tilde V(X)$ may not have a tree structure, so it is not a discretization. 

\subsubsection{Obtaining a tree structure}
We now modify $\tilde V(X)$ to a discretization $V(X)$.
We again proceed by induction.
For $n > N$, let $V_n(X) = V_n^0(X)$.
Now suppose that $n \leq N$ and we have constructed $(V_m(X))_{m \geq n + 1}$, to be a discretization of $X$. 
For each element $T \in \tilde{V}_n(X)$, we define subsets $\mathcal{C}(T)$ of $V_{n + 1}(X)$ as follows:
\begin{itemize}
    \item $\mathcal{C}(T)$ are all disjoint and their disjoint union is all of $V_{n + 1}(X)$.
    \item If $S \in V_{n + 1}(X)$ and $S \subseteq T$, then $S \in \mathcal{C}(T)$.
    \item If $S \in V_{n + 1}(X)$ intersects multiple $T$, then we pick one $T$ for which $S$ lies in $\mathcal{C}(T)$.
\end{itemize}
We now define $V_n(X) = \{\bigcup_{S \in \mathcal{C}(T)} S: T \in \tilde{V}_n(X)\}$. Thus, for each $I \in V_n(X)$, there exists an element $T \in \tilde{V}_n(X)$ such that
$$\dist(\partial T, \partial I) \leq 2L^{-n-1} \leq \frac{1}{10} L^{-n-2/3}$$
(where the second inequality is because $L \geq 10^3$), and for $x \in T$ satisfying (\ref{good cube condition 2}), $x \in I$.
Then for every $x \in X$ there exists a unique $I' \in V_{n + 1}(X)$ containing $x$ by our inductive assumption, and a unique $I \in V_n(X)$ which is a superset of $I'$, by the fact that $\{\mathcal C(T): T \in \tilde V_n(X)\}$ is a partition of $V_{n + 1}(X)$.
It follows that $(V_m(X))_{m \geq n}$ is a discretization of $X$.

By construction, there exists $x_0 \in X \cap T$ satisfying (\ref{good cube condition 2}), hence 
$$\dist(x_0, \partial I) \geq \dist(x_0, \partial T) - \dist(\partial I, \partial T) \geq \left(\frac{1}{5} - \frac{1}{10}\right) L^{-n-2/3} = \frac{1}{10} L^{-n-2/3}$$
and hence $x_0$ satisfies (\ref{good cube condition}).
If we denote by $I^0$ the cube that we modified to create $T$, then by Invariant \ref{closeness invariant},
$$\dist(\partial I^0, \partial I) \leq \dist(\partial I^0, \partial T) + \dist(\partial T, \partial I) \leq \left(\frac{1}{2} + \frac{1}{10}\right) L^{-n-2/3},$$
which one can use to show (\ref{perturbed dist is perturbation}). 
This completes the proof of Proposition \ref{discret}.

\subsection{Regularity of the discretization}
We now show that if the compact set $X$ is the support of a doubling measure, then its perturbed standard discretization $V(X)$ satisfies regularity conditions similar to those established in \cite[Lemma 2.1]{Dyatlov_2018} for the merged discretization in the case $d = 1$.

We begin by showing that every pair of tiles $(I, J) \in V_n(X) \times V_m(Y)$ have children which contain points for which the analogue 
$$|\Phi(x_0, y_0) - \Phi(x_0, y_1) - \Phi(x_1, y_0) + \Phi(x_1, y_1)| \gtrsim |x_0 - x_1| \cdot |y_0 - y_1|$$
of the reverse Cauchy--Schwarz inequality for the indefinite inner product $\partial^2_{xy} \Phi$ holds.
This is the key new estimate needed in the higher-dimensional case:

\begin{figure}[h]

\centering
\includegraphics[scale=0.7]{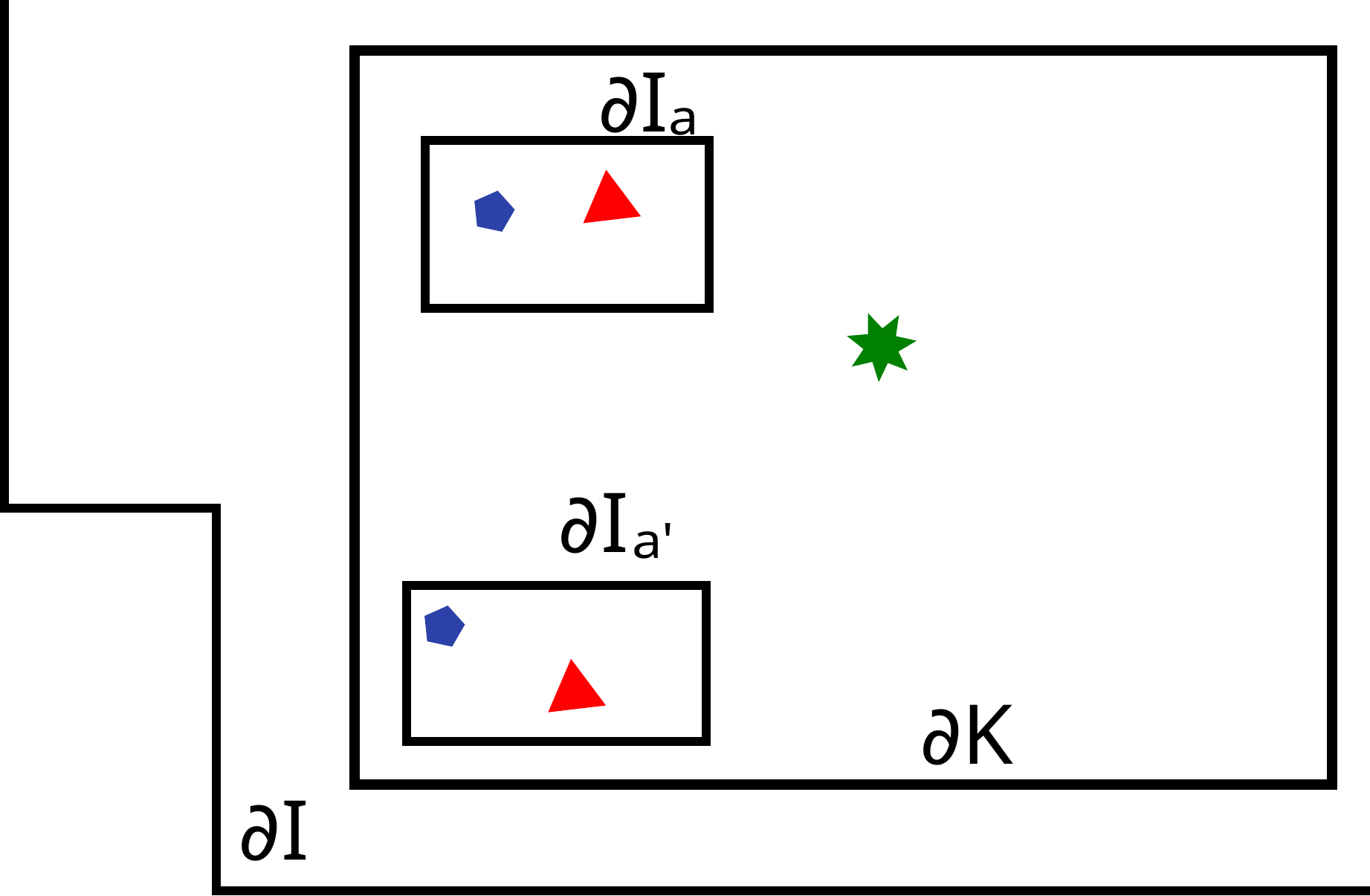}
\caption{A typical situation in the proof of Lemma \ref{DJ213}. The child tiles $I_a, I_{a'}$ are contained in the cube $K \subset I$, and are much smaller than $I$. The green $7$-gon denotes $\underline x$, the red triangles denote $\tilde x_a, \tilde x_{a'}$, and the blue pentagons denote $x_a, x_{a'}$.}
\label{figure3}
\end{figure}

\begin{lemma}\label{DJ213}
Let $\Phi \in C^3(\RR^d \times \RR^d)$, and let $X, Y \subseteq \RR^d$ be $\Phi$-nonorthogonal with constant $c_N$ from scales $(L^{-K_X},L^{-K_Y})$ to $1$.
Let $V(X), V(Y)$ be the perturbed standard discretizations of $X, Y$.
Then for
\begin{equation}\label{first estimate on L}
L \geq \max(180^3,  10^{10} c_N^{-3} \|\partial^2_{xy} \Phi\|_{C^1}^3) d^{3/2},
\end{equation}
and every $n < K_X$, $m < K_Y$, $I \in V_n(X)$, $J \in V_m(Y)$, there exist children $I_a, I_{a'}$ of $I$ and $J_b, J_{b'}$ of $J$ such that for every $x_\alpha \in I_\alpha$, $y_\beta \in J_\beta$, and $\omega_{\alpha \beta} := \Phi(x_\alpha, y_\beta)$,
we have the \dfn{reverse Cauchy--Schwarz inequality}
\begin{equation}\label{RCS}
    \frac{c_N}{1000}\leq L^{m+n+4/3} |\omega_{ab} - \omega_{a'b} - \omega_{ab'} + \omega_{a'b'}| \leq \frac{\|\partial^2_{xy} \Phi\|_{C^0}}{20}.
\end{equation}
and the \dfn{even spacing condition}
\begin{equation}\label{even spacing}
  L^{n + 2/3} |x_a - x_{a'}|, L^{m + 2/3} |y_b - y_{b'}| \leq \frac{1}{2}.
\end{equation}
Moreover, we may assume
\begin{equation}\label{line segment_1}
    \text{ for any }\, 
    x_a\in I_a, x_{a'}\in I_{a'},\,\text{ the line segment } \, \overline{x_a x_{a'}}\, \text{ always lies in }\, I.
\end{equation}
\end{lemma}
\begin{proof}
By Proposition \ref{discret}, we may choose $\underline x \in X\cap I$ and $\underline y \in Y\cap J$ such that
$$\min\left(L^{n + 2/3} \dist(\underline x, \partial I), L^{m + 2/3} \dist(\underline y, \partial J)\right) \geq \frac{1}{10}.$$
Let $r_X=\frac{1}{20}L^{-n-2/3}$, $r_Y=\frac{1}{20}L^{-m-2/3}$.
One can show that if (\ref{first estimate on L}) holds, then $L \geq 20^3$ and 
\begin{equation}\label{annoying factors}
    (1 + 2L^{-2/3})^4 \leq \frac{9}{8}.
\end{equation}
Since $n \leq K_X - 1$ and $L \geq 20^3$,
$$r_X = \frac{1}{20} L^{-n-2/3} \geq \frac{1}{20} L^{-K_X} L^{1/3} \geq L^{-K_X},$$
and similarly $r_Y \geq L^{-K_Y}$.
So by nonorthogonality, there exist $\tilde x_a, \tilde x_{a'} \in X\cap B(\underline x,r_X)$ and $\tilde y_b, \tilde y_{b'} \in Y\cap B(\underline y,r_Y)$ such that for $\tilde \omega_{\alpha \beta} := \Phi(\tilde x_\alpha, \tilde y_\beta)$,
\begin{equation}\label{tilde nonorthogonality}
    |\tilde \omega_{ab} - \tilde \omega_{a'b} - \tilde \omega_{ab'} + \tilde \omega_{a'b'}| \geq c_N r_X r_Y.
\end{equation}
In the other direction, (\ref{geometric MVT}) and the triangle inequality gives
\begin{equation}\label{tilde CS}
    |\tilde \omega_{ab} - \tilde \omega_{a'b} - \tilde \omega_{ab'} + \tilde \omega_{a'b'}| \leq \|\partial^2_{xy} \Phi\|_{C^0} \cdot |\tilde x_a - \tilde x_{a'}| \cdot |\tilde y_b - \tilde y_{b'}|.
\end{equation}
Let $I_\alpha$ the the children of $I$ containing $\tilde x_\alpha$ and $J_\beta$ be the children of $J$ containing $y_\beta$. Pick arbitrary points $x_\alpha \in I_\alpha$ and $y_\beta \in J_\beta$. We first use (\ref{perturbed dist is perturbation}), (\ref{annoying factors}), and (\ref{first estimate on L}) to bound
\begin{align*}
    |x_a - x_{a'}| 
    &\leq 2r_X + \diam I_a + \diam I_{a'} \\
    &\leq \frac{1}{10} L^{-n-2/3} + 2 d^{1/2} L^{-n-1}(1 + 2L^{-2/3})^2 \\
    &\leq \frac{1}{10} L^{-n-2/3} + 5 d^{1/2} L^{-n - 1} \\
    &\leq \frac{1}{2} L^{-n-2/3}.
\end{align*}
A similar estimate holds on $|y_b - y_{b'}|$, which proves the upper bound in (\ref{even spacing}).

To prove (\ref{RCS}), let $c_x := 2r_X$, $c_y := 2r_Y$, $\varepsilon_x := \max(\diam I_a, \diam I_{a'})$, and $\varepsilon_y := \max(\diam J_b, \diam J_{b'})$.
Then by (\ref{geometric MVT}), (\ref{difference of MVTs}), (\ref{perturbed dist is perturbation}), (\ref{annoying factors}), and (\ref{first estimate on L}),
 \begin{align*}
    &|\omega_{ab} - \omega_{ab'} - \omega_{a'b} + \omega_{a'b'} - \tilde \omega_{ab} + \tilde \omega_{ab'} + \tilde \omega_{a'b} - \tilde \omega_{a'b'}| \\
    &\qquad \leq 7 \|\partial^2_{xy} \Phi\|_{C^1} (c_x \varepsilon_y + c_y \varepsilon_x) \\
    &\qquad \leq \frac{7}{5} \|\partial^2_{xy} \Phi\|_{C^1} (d^{1/2} L^{-n-m-5/3}(1 + 2L^{-2/3})^2) \\
    &\qquad \leq 2 \|\partial^2_{xy} \Phi\|_{C^1} d^{1/2} L^{n-m-5/3}.
\end{align*}
Combining this estimate with (\ref{tilde nonorthogonality}) and (\ref{first estimate on L}), we obtain
\begin{align*}
    &|\omega_{ab} - \omega_{ab'} - \omega_{a'b} + \omega_{a'b'}| \\
    &\qquad \geq |\tilde \omega_{ab} - \tilde \omega_{a'b} - \tilde \omega_{ab'} + \tilde \omega_{a'b'}| \\
    &\qquad \qquad - |\omega_{ab} - \omega_{ab'} - \omega_{a'b} + \omega_{a'b'} - \tilde \omega_{ab} + \tilde \omega_{ab'} + \tilde \omega_{a'b} - \tilde \omega_{a'b'}| \\
    &\qquad \geq \frac{c_N}{400} L^{-n-m-4/3} - 2 \|\partial^2_{xy} \Phi\|_{C^1} d^{1/2} L^{-n-m-5/3} \\
    &\qquad \geq \frac{c_N}{1000} L^{-n-m-4/3}
\end{align*}
which is the desired lower bound in (\ref{RCS}).
For the upper bound, since $\tilde{x}_a,\tilde{x}_{a'}\in B(\underline{x},r_X)$ and $\tilde{y}_b,\tilde{y}_{b'}\in B(\underline{y},r_Y)$ we use \eqref{geometric MVT}:
\begin{align*}
   |\omega_{ab} - \omega_{ab'} - \omega_{a'b} + \omega_{a'b'}| &\leq  4\|\partial^2_{xy} \Phi\|_{C^0} (r_X+\sqrt{d}L^{-n-1}) (r_Y+\sqrt{d}L^{-m-1})\\
    &\leq \frac{\|\partial^2_{xy} \Phi\|_{C^0}}{20} L^{-n-m-4/3}.
\end{align*}

Finally we prove (\ref{line segment_1}). We use (\ref{perturbed dist is perturbation}), (\ref{annoying factors}), (\ref{first estimate on L}), and the fact that $\dist(a, b) \leq |a - b|$ to estimate 
$$\dist(x_a, \underline x) \leq \dist(x_a, \tilde x_a) + \dist(\tilde x_a, \underline x) \leq 2L^{-n-1} + r_X \leq \frac{1}{15} L^{-n-2/3}.$$
The same bound holds for $x_{a'}$ and it follows that $x_a, x_{a'}$ are contained in the convex set $K := B_\infty(\underline x, L^{-n-2/3}/15)$.
In particular, $\ell := \overline{x_a x_{a'}}$ satisfies $\ell \subset K$.
This implies $\ell \subset I$, since
$$\dist(\partial K, \partial I) \geq \dist(\underline x, \partial I) - \frac{L^{-n-2/3}}{15} \geq \frac{L^{-n-2/3}}{30}$$
so that $K \subseteq I$.
\end{proof}

We now give a probabilistic interpretation of the above lemmata. 
To establish notation, suppose that $I \in V_n(X)$ for some compact set $X$ and some $n$.
We write $\{I_a: a \in A\}$ for the set of children of $I$.
This induces the structure of a probability space on $A$: namely,
$$\Pr(a) := \frac{\mu_X(I_a)}{\mu_X(I)}.$$

\begin{proposition}\label{random discretization}
Let $\Phi \in C^3(\RR^d \times \RR^d)$, and suppose that $L$ satisfies (\ref{first estimate on L}). Let $(X, \mu_X)$ be doubling with constant $C_D(X)$ on scales $[ L^{-K_X},1]$, let $(Y, \mu_Y)$ be doubling with constant $C_D(Y)$ on scales $[ L^{-K_Y},1]$, let $V(X), V(Y)$ be their perturbed standard discretizations, and assume that $(X, Y)$ is $\Phi$-nonorthogonal with constant $c_N$ from scales $(L^{-K_X}, L^{-K_Y})$ to $1$, $n < K_X$, $m < K_Y$, $I \in V_n(X)$, and $J \in V_m(Y)$, and $\{I_a: a \in A\}$ and $\{J_b: b \in B\}$ the sets of children of $I, J$. Furthermore, choose for each $a \in A$ and $b \in B$, $x_a \in I_a$ and $y_b \in J_b$, and set $\omega_{ab} := \Phi(x_a, y_b)$.
 
Draw independent random outcomes $a, a' \in A$ and $b, b' \in B$.
Then with probability
\begin{equation}\label{probability estimate}
    \rho \geq C_D(X)^{-2\lceil \log_2(20L^{5/3}) \rceil} C_D(Y)^{-2\lceil \log_2(20L^{5/3}) \rceil},
\end{equation}
we have the reverse Cauchy--Schwarz inequality
\begin{equation}\label{Reverse Cauchy Schwarz}
    \frac{c_N}{1000} L^{-1/3} \leq L^{n + m + 1} |\omega_{ab} - \omega_{a'b} - \omega_{ab'} + \omega_{a'b'}|\leq \pi
\end{equation}
and the even spacing condition
\begin{equation}\label{comparable lengths}
  L^{n + 2/3} |x_a - x_{a'}|, L^{m + 2/3} |y_b - y_{b'}| \leq \frac{1}{2}.
\end{equation}
Moreover, we may assume
\begin{equation}\label{eq:mean value condition}
    \text{ for any }\, 
    x_a\in I_a, x_{a'}\in I_{a'},\,\text{ the line segment } \, \overline{x_a x_{a'}}\, \text{ always lies in }\, I.
\end{equation}
\end{proposition}
\begin{proof}

By Lemma \ref{DJ213}, there exist $a,b,a', b'$ satisfying \eqref{RCS} and \eqref{even spacing}.
By definition of the perturbed standard discretization, there exists $x_*\in I_a \cap X$ with $I_*:=\frac{1}{10}B_\infty(x_*,L^{-n-5/3})\subset I_a$. Moreover, $I\subset B_\infty(x_0, 2L^{-n})=I_*(20L^{5/3})$. Therefore,

$$\Pr(a) = \frac{\mu_X(I_a)}{\mu_X(I)} \geq \frac{\mu_X(I_*)}{\mu_X(I_*(20L^{5/3}))} \geq C_D(X)^{-\lceil \log_2(20L^{5/3}) \rceil}.$$

We have analogous lower bound on $\Pr(b),\Pr(a'),\Pr(b')$.
Then by independence,
$$\rho \geq \Pr(a) \Pr(a') \Pr(b) \Pr(b'),$$
which gives (\ref{probability estimate}), and (\ref{RCS}) and (\ref{even spacing}) clearly imply (\ref{comparable lengths}) and the lower bound on (\ref{Reverse Cauchy Schwarz}). The condition \eqref{eq:mean value condition} comes from \eqref{line segment_1}.
For the upper bound we apply (\ref{RCS}) and (\ref{first estimate on L}).
\end{proof}

\section{The induction on scales}\label{induction on scale}
We now begin the proof of Theorem \ref{main theorem}.
Let $\Phi \in C^3(\RR^d \times \RR^d)$ and $p \in C^1(\RR^d \times \RR^d)$ be the phase and symbol of $\mathcal B_h$, and let $K := \lfloor -\log_L h \rfloor$.

Let $(X, \mu_X)$ and $(Y, \mu_Y)$ be doubling with constants $C_D(X), C_D(Y)$ on scales $\geq h$, let $V(X), V(Y)$ be their perturbed standard discretizations, and assume that $(X, Y)$ is $\Phi$-nonorthogonal with constant $c_N$ from scales $(h, h)$ to $1$.

For $I \in V_n(X)$ and $J \in V_m(Y)$, where $n + m + 1 = K$, we set 
$$F_J(x) = \frac{1}{\mu_Y(J)} \int_J \exp\left(i\frac{\Phi(x, y) - \Phi(x, y_J)}{h}\right) p(x, y) f(y) \dif \mu_Y(y).$$
Here $y_J$ is the center of $J^0$, the box in the standard discretization associated to $J$.
Let $\{I_a: a \in A\}$ and $\{J_b: b \in B\}$ be sets of children with their usual probability measures.
Let $x_a := \argmax_{I_a} |F_J|$ and $y_b := y_{J_b}$.

\subsection{Mean value space}
We need to generalize the space $C_\theta(I)$ where $d = 1$ (see \cite[\S2.2]{Dyatlov_2018} and also \cite[Lemma 5.4]{Naud2005}), which is supposed to locally measure oscillation on $I$ whilst also being ``scale-invariant."\footnote{We cannot use the space $C^1(I)$ with its norm $\|f\|_{C^1(I)} := \|f\|_{C^0(I)} + \|\nabla f\|_{C^0(I)}$, because the first and second terms in the norm will scale differently if we rescale $I$.} This will allow us to get some gain out of the cancellation obtained from nonorthogonality while performing induction on scales.

\begin{definition}
Given $I \in V_n(X)$ and $\theta \in (0, 1)$, we define the $C_\theta(I)$ norm for functions $f \in C^1(I)$ by
$$\|f\|_{C_\theta(I)} := \max\left(\|f\|_{C^0(I)}, \theta \diam(I) \|\nabla f\|_{C^0(I)}\right).$$
\end{definition}

Given $J \in V_m(Y)$, we set $\Psi_b: I \to \mathbf{R}$ as
$$\Psi_b(x) := \frac{\Phi(x, y_{J_b}) - \Phi(x, y_J)}{h}.$$

\begin{lemma}\label{theta assumptions}
Let $\theta \leq \frac{1}{8 \max(1, \|\partial^2_{xy} \Phi\|_{C^0(I_{conv})})}$ (where $I_{conv}$ is the convex hull of $I$) and $L \geq 10$. Then for $f \in C_\theta(I)$,
\begin{equation}\label{assumption on theta}
\|e^{i\Psi_b} f\|_{C_\theta(I_a)} \leq \|f\|_{C_\theta(I)},
\end{equation}
\end{lemma}

\begin{proof}
Observe that if $\psi$ is a smooth function on $I_{a,conv}$, then any $f \in C_\theta(I_a)$ satisfies
$$|\nabla (e^{i\psi} f)| = |ie^{i\psi} f \nabla \psi + e^{i\psi} \nabla f| \le |f \nabla \psi| + |\nabla f|.$$
Hence
$$\theta \diam(I_a) |\nabla (e^{i\Psi_b} f)(x)| \le \theta \diam(I_a) \|\nabla \Psi_b\|_{C^0(I_a)} \|f\|_{C^0(I_a)} + \theta \diam(I_a) \|\nabla f\|_{C^0(I_a)}.$$
We estimate that
\begin{align*}
|\nabla \Psi_b(x)| &= \frac{1}{h}|\partial_x(\Phi(x, y_{J_b}) - \Phi(x, y_J))| \le \frac{1}{h} |y_{J_b} - y_J|\|\partial_{xy}^2 \Phi\|_{C^0(I_{conv})} \\
&\leq \frac{\diam(J)}{h} \|\partial^2_{xy} \Phi\|_{C^0(I_{conv})}.
\end{align*}
So by hypothesis on $\theta$ and $L$,
\begin{align*}
    \theta \diam(I_a) \|\nabla \Psi_b\|_{C^0(I_a)} \|f\|_{C^0(I_a)} 
    &\leq \theta\frac{\diam(I_a) \diam(J)}{h} \|\partial^2_{xy} \Phi\|_{C^0(I_{conv})} \|f\|_{C^0(I)} \\
    &\leq \theta (1 + L^{-2/3})^2 \|\partial^2_{xy} \Phi\|_{C^0(I_{conv})} \|f\|_{C_\theta(I)} \\
    &\leq \frac{\|f\|_{C_\theta(I)}}{4}.
\end{align*}
In addition, by hypothesis on $L$,
$$\theta \diam(I_a) \|\nabla f\|_{C^0(I_a)} \le \frac{2}{L} \theta \diam(I) \|\nabla f\|_{C^0(I)} \leq \frac{\|f\|_{C_\theta(I)}}{4}.$$
Summing up,
$$\theta \diam(I_a) \|\nabla(e^{i\Psi_b} f)\|_{C^0(I_a)} \leq \|f\|_{C_\theta(I)}.$$
We also trivially have
$$\|f\|_{C^0(I_a)} \leq \|f\|_{C^0(I)} \leq \|f\|_{C_\theta(I)},$$
which proves (\ref{assumption on theta}).
\end{proof}

\subsection{Inductive step}
Our next task is to prove the following analogue of \cite[Lemma 3.2]{Dyatlov_2018}.

\begin{proposition}\label{statistical iteration}
Let $I \in V_n(X)$, $J \in V_m(Y)$, where $n + m + 1 = K$.
Draw a random $b \in B$, and assume that (\ref{Reverse Cauchy Schwarz}) and (\ref{comparable lengths}) hold with probability $\rho$.
Assume that
\begin{align}
    L &\geq \max\left(\frac{10^{12} d^3}{c_N^3 \theta^{3/2}}, \frac{{10^{10} \|\partial^2_{xy} \Phi\|_{C^1}^3} d^{3/2}}{c_N^3}\right), \label{formula for L}\\
    \varepsilon_1 &\leq \frac{\rho^2 c_N^2}{10^9  d^2 L^{2/3}}. \label{formula for epsilon1}
\end{align}
Then we have the improvement
\begin{equation}\label{improved volume bound}
\Expect_{a \in A} \|F_J\|_{C_\theta(I_a)}^2 \leq (1 - \varepsilon_1) \Expect_{b \in B} \|F_{J_b}\|_{C_\theta(I)}^2.
\end{equation}
\end{proposition}

\subsubsection{The contradiction assumption}
We set up the proof of Proposition \ref{statistical iteration} by first recording 
\begin{equation}\label{twisting formula}
    F_J = \Expect_{b \in B} e^{i\Psi_b} F_{J_b}.
\end{equation}
We have the following lemma which is nearly identical to \cite[Lemma 3.3]{Dyatlov_2018}.

\begin{lemma}
For each $a \in A$,
\begin{equation}\label{Expectation bound}
\|F_J\|_{C_\theta(I_a)}^2 \leq \left(\Expect_{b \in B} \|F_{J_b}\|_{C_\theta(I)}\right)^2 \leq \Expect_{b \in B} \|F_{J_b}\|_{C_\theta(I)}^2.
\end{equation}
\end{lemma}
\begin{proof}
By (\ref{assumption on theta}), 
$$\|e^{i\Psi_b} F_{J_b}\|_{C_\theta(I_a)} \leq \|F_{J_b}\|_{C_\theta(I)}.$$
The assertions of (\ref{Expectation bound}) now follow from (\ref{twisting formula}) and the Cauchy--Schwarz inequality.
\end{proof}

We set $R := \Expect_{b \in B} \|F_{J_b}\|_{C_\theta(I)}^2$.
Draw $a \in A$ independently of $b$.
Taking expectations in (\ref{Expectation bound}), we obtain
\begin{equation}\label{Variance bound}
\sigma^2 := \Expect_{b \in B} \|F_{J_b}\|_{C_\theta(I)}^2 - \Expect_{a \in A} \|F_J\|_{C_\theta(I_a)}^2 \geq \Var_{b \in B} \|F_{J_b}\|_{C_\theta(I)}.
\end{equation}
In particular, (\ref{Variance bound}) can be written
$$\sigma^2 = R -\Expect_{a\in A} \|F_J\|_{C_\theta(I_a)}^2.$$
If we knew that $\sigma^2 \geq \varepsilon_1 R$, then the improvement (\ref{improved volume bound}) would follow.
So, we assume towards contradiction that
\begin{equation}\label{contradiction assumption}
\sigma^2 < \varepsilon_1 R.
\end{equation}

Let $F_{ab} := F_{J_b}(x_a)$, $\omega_{ab} := \Psi_b(x_a)$, and $f_{ab} := e^{i\omega_{ab}} F_{ab}$, so that
\begin{equation}\label{expectation fab}
F_J(x_a) = \Expect_{b \in B} f_{ab}
\end{equation}
and for each $a \in A$,
\begin{equation}\label{bounding Fab}
\Expect_{b \in B} |F_{ab}|^2 \leq \Expect_{b \in B} \|F_{J_b}\|_{C_\theta(I)}^2 = R.
\end{equation}

\subsubsection{Outline of the proof}
By our contradiction assumption (\ref{contradiction assumption}) and variance bound (\ref{Variance bound}), the $C_\theta(I)$ norms of the functions $F_{J_b}$ are all almost independent of $b$.
One can show that $f_{ab}$ is almost independent of $b$ (see \eqref{lower bound on expected argmax}). By the mean value theorem, $F_{ab}$ does not vary too much in $a$ (see \eqref{discrepancy upper bound}). 
However, the events (\ref{Reverse Cauchy Schwarz}) and (\ref{comparable lengths}) have positive probability, so we may condition on them without losing too much, and after conditioning, the phases of $f_{ab}$ and $f_{a'b'}$ cannot be too correlated by (\ref{Reverse Cauchy Schwarz}) and (\ref{comparable lengths}).
So we expect cancellation between $f_{ab}$ and $f_{a'b'}$ whenever $a,a',b,b'$ are drawn at random, by the square-root cancellation heuristic.
This cancellation implies that the conditional expectation of $|F_{ab}|^2$ is both very small and comparable to $R$, a contradiction.

\subsubsection{Two unconditional moment estimates}
We now make two unconditional moment estimates; we shall later use Cantelli's inequality to show that weaker versions of the same moment estimates hold even when we condition on the events (\ref{Reverse Cauchy Schwarz}) and (\ref{comparable lengths}).

\begin{lemma}\label{variance and iids claim}
One has
\begin{align}
\Expect_{a \in A} \Var_{b \in B} f_{ab} &\leq \Expect_{\substack{a \in A \\ b \in B}} |F_{ab}|^2 -R+2\sigma^2\leq 2\sigma^2, \label{lower bound on expected argmax} \\
\Expect_{\substack{a \in A \\ b \in B}} |F_{ab}| &\geq (1 - 2\varepsilon_1) \sqrt R. \label{lower bound on EF}
\end{align}
\end{lemma}
\begin{proof}
We follow \cite[Lemma 3.5]{Dyatlov_2018}.
By Lemma \ref{theta assumptions}, for each $a, b$,
$$\theta \|\nabla (e^{i\Psi_b}F_{J_b})\|_{C^0(I_a)} \diam I_a \leq \frac{\|F_{J_b}\|_{C_\theta(I)}}{2}.$$
From the definition of $C_\theta(I_a)$, (\ref{twisting formula}), and the triangle inequality, for each $a \in A$,
\begin{align*}
    \|F_J\|_{C_\theta(I_a)} &= \max\left(\|F_J\|_{C^0(I_a)}, \theta \|\nabla F_J\|_{C^0(I_a)} \diam I_a\right) \\
    &\leq \max\left(\|F_J\|_{C^0(I_a)}, \theta \diam I_a \Expect_{b \in B} \|\nabla(e^{i\Psi_b} F_{J_b})\|_{C^0(I_a)}\right) \\
    &\leq \max\left(\|F_J\|_{C^0(I_a)}, \frac{1}{2} \Expect_{b \in B} \|F_{J_b}\|_{C_\theta(I)}\right).
\end{align*}
We estimate the squares of the two terms in the maximum using (\ref{Expectation bound}):
$$\|F_J\|_{C^0(I_a)}^2 \leq \frac{1}{2} (\|F_J\|_{C^0(I_a)}^2 + \|F_J\|_{C_\theta(I_a)}^2) \leq \frac{1}{2}(\|F_J\|_{C^0(I_a)}^2 + R),$$
and
$$\left(\frac{1}{2} \Expect_{b \in B} \|F_{J_b}\|_{C_\theta(I)}\right)^2 \leq \frac{1}{4} \Expect_{b \in B} \|F_{J_b}\|_{C_\theta(I)}^2 \leq \frac{R}{4} \leq \frac{1}{2}(\|F_J\|_{C^0(I_a)}^2 + R).$$
In summary, we have 
\begin{equation}\label{reverse bound on theta}
\|F_J\|_{C_\theta(I_a)}^2 \leq \frac{1}{2}(\|F_J\|_{C^0(I_a)}^2 + R).
\end{equation}
After taking expectations and applying (\ref{Variance bound}), we get 
\begin{equation}\label{R minus 2sigma}
\Expect_{a \in A} \|F_J\|_{C^0(I_a)}^2 \geq 2 \Expect_{a \in A} \|F_J\|_{C_\theta(I_a)}^2 - R = R - 2\sigma^2.
\end{equation}
We also record that, by (\ref{twisting formula}), (\ref{expectation fab}), and the fact that $x_a$ maximizes $|F_J|$,
\begin{equation}\label{a|b}
    \Expect_{a \in A} \left|\Expect_{b \in B} f_{ab}\right|
    = \Expect_{a \in A} \left|\Expect_{b \in B} e^{i\Psi_b(x_a)} F_{J_b}(x_a)\right|
    = \Expect_{a \in A} |F_J(x_a)|
    = \Expect_{a \in A} \|F_J\|_{C^0(I_a)}.
\end{equation}
Combining this fact with (\ref{R minus 2sigma}),
$$\Expect_{a \in A} \left|\Expect_{b \in B} f_{ab}\right|^2 \geq R - 2\sigma^2.$$
Therefore
\begin{align}\label{we get a R - 2sigma 1}
    \Expect_{\substack{a \in A \\ b \in B}} |F_{ab}|^2 &= \Expect_{\substack{a \in A \\ b \in B}} |f_{ab}|^2 = \Expect_{a \in A} \left(\left|\Expect_{b \in B} f_{ab}\right|^2 + \Var_{b \in B} f_{ab}\right) \\
    &\geq R - 2\sigma^2 + \Expect_{a \in A} \Var_{b \in B} f_{ab}. \label{we get a R - 2sigma 2}
\end{align}
Rearranging, we obtain
$$\Expect_{a \in A} \Var_{b \in B} f_{ab} \leq \Expect_{\substack{a \in A \\ b \in B}} |F_{ab}|^2 - R+2\sigma^2.$$
Then (\ref{lower bound on expected argmax}) follows from (\ref{bounding Fab}).

To obtain (\ref{lower bound on EF}), we first estimate 
\begin{align*}
    \Expect_{\substack{a \in A \\ b \in B}} |F_{ab}|=\Expect_{\substack{a \in A \\ b \in B}} |f_{ab}|\geq\Expect_{a\in A}\left|\Expect_{b\in B}f_{ab}\right|=\Expect_{a\in A}\sqrt{\Expect_{b\in B}|f_{ab}|^2-\Var_{b\in B}f_{ab}}.
\end{align*}
From \eqref{we get a R - 2sigma 1}-\eqref{we get a R - 2sigma 2} and (\ref{bounding Fab}), and the contradiction assumption (\ref{contradiction assumption}), 
\begin{align*}
    \Expect_{a\in A}\sqrt{\Expect_{b\in B}|f_{ab}|^2-\Var_{b\in B}f_{ab}} 
    &\geq \frac{1}{\sqrt{\max_{a \in A} \Expect_{b \in B} |f_{ab}|^2}} \Expect_{a \in A} \left(\Expect_{b \in B} |f_{ab}|^2 - \Var_{b \in B} f_{ab}\right) \\
    &\geq \Expect_{a\in A}\frac{\Expect_{b\in B}|f_{ab}|^2-\Var_{b\in B}f_{ab}}{\sqrt{R}}\geq \frac{R-2\sigma^2}{\sqrt{R}}\\
    &\geq (1-2\varepsilon_1)\sqrt{R}. \qedhere
\end{align*}
\end{proof}

\subsubsection{Drawing random nonorthogonal tiles}
By (\ref{Expectation bound}) and the Cauchy--Schwarz inequality,
\begin{equation}\label{Expectation bound 2}
\Expect_{b \in B} \|F_{J_b}\|_{C_\theta(I)} \leq \sqrt R.
\end{equation}
Let $T$ be the event that $\|F_{J_b}\|_{C_\theta(I)} \leq 2 \sqrt R$.
By the moment bounds (\ref{Expectation bound 2}) and (\ref{Variance bound}), the contradiction assumption (\ref{contradiction assumption}), and Cantelli's inequality (\ref{Cantelli}),
\begin{equation}\label{probability lower bound}
\Pr(T) > 1 - \varepsilon_1.
\end{equation}
We let $T'$ be the respective event for $b'$, where $a', b'$ are drawn independently from $a, b$.
From (\ref{lower bound on expected argmax}), (\ref{probability lower bound}), and (\ref{variance in a product space}), we obtain
\begin{align}
\Expect_{\substack{a \in A \\ b, b' \in B}} (|f_{ab} - f_{ab'}|^2|T \cap T')
&\leq \frac{1}{\Pr(T \cap T')} \Expect_{\substack{a \in A \\ b, b' \in B}} |f_{ab} - f_{ab'}|^2 \label{upper bound on oscillating discrepancy 1} \\
&\leq \frac{2}{\Pr(T)^2} \Expect_{a \in A} \Var_{b \in B} f_{ab} \leq 2.5 \cdot 2\sigma^2 = 5\sigma^2 \label{upper bound on oscillating discrepancy 2}.
\end{align}

If $T$ and \eqref{eq:mean value condition} hold, then by Lemma \ref{theta assumptions},
\begin{equation}\label{discrepancy upper bound}
|F_{ab} - F_{a'b}| \leq \frac{2\sqrt R}{\theta} L^{H(I)} |x_a - x_{a'}|
\end{equation}
Let $S$ be the intersection of $T$, $T'$, and the events \eqref{Reverse Cauchy Schwarz}, \eqref{comparable lengths} and \eqref{eq:mean value condition}.
By (\ref{formula for epsilon1}), $\varepsilon_1 \leq \rho/10$, so by (\ref{probability lower bound}),
\begin{equation}\label{probability lower bound 2}
\frac{\Pr(S)}{\Pr(T)^2} \geq \frac{\rho - 2(1 - \Pr(T))}{\Pr(T)^2} \geq \frac{\rho - 2\varepsilon_1}{(1 - \varepsilon_1)^2} \geq \frac{\rho}{2}.
\end{equation}
If $S$ holds, then by (\ref{discrepancy upper bound}) and (\ref{comparable lengths}),
\begin{equation}\label{discrepancy upper bound 2}
|F_{ab} - F_{a'b}| \leq \frac{\sqrt R}{L^{2/3} \theta}.
\end{equation}

\subsubsection{Conditional second moment bounds}
We now use (\ref{probability lower bound 2}) and (\ref{discrepancy upper bound 2}) to obtain lower and upper bounds on $\Expect(|F_{ab}|^2|S)$ which are not both tenable.

\begin{lemma}
For $M := 8000000$,
\begin{equation}\label{bound on Fab}
\Expect_{\substack{a \in A \\ b \in B}} (|F_{ab}|^2|S) \leq Md^2\left(\frac{R}{c_N^2 L^{2/3} \theta} + 2 \frac{L^{2/3} \sigma^2}{c_N^2 \rho}\right).
\end{equation}
\end{lemma}
\begin{proof}
We take all expectations and probabilities over $a, a', b, b'$.
Write
$$\tau := \omega_{ab} - \omega_{ab'} - \omega_{a'b} + \omega_{a'b'},$$
so if $S$ holds then
$$|e^{i\tau} - 1|^2 \geq |\tau|^2 \geq 10^{-6} c_N^{2} L^{-2/3}$$
by (\ref{Reverse Cauchy Schwarz}) and \cite[Lemma 2.6]{Dyatlov_2018}.
Following \cite[19]{Dyatlov_2018}, we rewrite 
\begin{align*}
|(e^{i\tau} - 1)F_{ab}|
&= |e^{i(\omega_{ab} - \omega_{ab'})} F_{ab} - e^{i(\omega_{a'b} - \omega_{a'b'})} F_{ab}| \\
&= |e^{-i\omega_{ab'}} (f_{ab} - f_{ab'}) + F_{ab'} - F_{a'b'} - e^{-i\omega_{a'b'}}(f_{a'b} - f_{a'b'}) + F_{ab} - F_{ab'}|.
\end{align*}
So by the triangle inequality in $L^2$,
$$\Expect(|(e^{i\tau} - 1)F_{ab}|^2|S) \leq 4 \Expect(|F_{ab} - F_{a'b}|^2 + |F_{a'b'} - F_{ab'}|^2 + |f_{ab} - f_{ab'}|^2 + |f_{a'b'} - f_{a'b}|^2|S).$$
So 
\begin{align*}
\Expect(|F_{ab}|^2|S) &\leq 10^6 \cdot \frac{d^2 L^{2/3}}{c_N^2} \Expect(|(e^{i\tau} - 1)F_{ab}|^2|S) \\
&\leq \frac{Md^2 L^{2/3}}{2c_N^2} \Expect(|F_{ab} - F_{a'b}|^2 + |F_{a'b'} - F_{ab'}|^2|S)\\
&\qquad + \frac{Md^2 L^{2/3}}{2c_N^2} \Expect(|f_{ab} - f_{ab'}|^2 + |f_{a'b'} - f_{a'b}|^2|S).
\end{align*}
Applying (\ref{discrepancy upper bound 2}),
$$|F_{ab} - F_{a'b}|^2 + |F_{a'b'} - F_{ab'}|^2 \leq \frac{2 R}{L^{4/3} \theta}.$$
Since $S$ implies $T \cap T'$, and $a,a'$ are independent,
$$\Expect(|f_{ab} - f_{ab'}|^2 + |f_{a'b'} - f_{a'b}|^2|S) \leq 2 \frac{\Pr(T)^2}{\Pr(S)} \Expect(|f_{ab} - f_{ab'}|^2|T \cap T').$$
By (\ref{probability lower bound 2}), $\Pr(T)^2/\Pr(S) \leq 2/\rho$.
Summing all this up and applying (\ref{upper bound on oscillating discrepancy 1}-\ref{upper bound on oscillating discrepancy 2}), we conclude (\ref{bound on Fab}).
\end{proof}

\begin{lemma}
One has
\begin{equation}\label{bound on Fab 2}
\Expect_{\substack{a \in A \\ b \in B}} (|F_{ab}|^2|S) \geq \frac{R}{6}.
\end{equation}
\end{lemma}
\begin{proof}
By \eqref{lower bound on EF}, we conclude that
\begin{align*}
    \Pr_{a\in A}\left(\Expect_{b\in 
 B}|F_{ab}|<(1-2\sqrt{\epsilon_1})\sqrt{R}\right)\leq \sqrt{\epsilon_1}.
\end{align*}
By Cantelli's inequality (\ref{Cantelli}),
$$\Pr_{b\in B}\left(|F_{ab}| \leq \Expect_{b\in B} |F_{ab}| - \frac{\sqrt R}{2}\right) \leq \frac{\Var_{b\in B} |F_{ab}|}{\Var_{b\in B} |F_{ab}| + R/4}.$$
Since $|F_{ab}| = |f_{ab}|$, it follows from (\ref{lower bound on expected argmax}) and (\ref{contradiction assumption}) that 
\begin{align*}
&\Pr\left(|F_{ab}|^2 \leq \frac{R}{5}\right)\\
&\leq\Pr_{a\in A}\left(\Expect_{b\in 
 B}|F_{ab}|<(1-2\sqrt{\epsilon_1})\sqrt{R}\right)+\Pr\left(\Expect_{b\in  B}|F_{ab}|\geq (1-2\sqrt{\epsilon_1})\sqrt{R},\, |F_{ab}|^2\leq \frac{R}{5}\right)\\ 
 &\leq \sqrt{\epsilon_1}+\Expect_{a\in A}\Pr_{b\in B}\left(|F_{ab}| \leq \Expect_{b\in B} |F_{ab}| - \frac{\sqrt R}{2}\right) \\
&\leq \sqrt{\epsilon_1}+\frac{4\Expect_{a\in A}\Var_{b\in B} f_{ab}}{R} \\
&\leq\sqrt{\epsilon_1}+ \frac{8\sigma^2}{R} < 2\sqrt{\epsilon_1}.
\end{align*}
But by (\ref{probability lower bound 2}),
$$\Pr\left(|F_{ab}|^2 \leq \frac{R}{5}\bigg|S\right) = \frac{\Pr((|F_{ab}|^2 \leq R/5) \cap S)}{\Pr(S)} \leq \frac{2\Pr(|F_{ab}|^2 \leq R/5)}{\rho}.$$
The definition (\ref{formula for epsilon1}) of $\varepsilon_1$ then implies 
$$\Pr\left(|F_{ab}|^2 \leq \frac{R}{5}\bigg|S\right) \leq \frac{4\sqrt{\varepsilon_1}}{\rho} < L^{-1/3}.$$
Therefore 
$$\Pr\left(|F_{ab}|^2 \geq \frac{R}{5}\bigg|S\right) \geq 1 - L^{-1/3},$$
so by Markov's inequality and the assumption (\ref{formula for L}),
\begin{align*}
    \Expect_{\substack{a \in A \\ b \in B}}(|F_{ab}|^2|S) &\geq \frac{R}{5} \Pr\left(|F_{ab}|^2 \geq \frac{R}{5}\bigg|S\right) \geq \frac{R}{6}. \qedhere
\end{align*}

\end{proof}

\subsubsection{Deriving a contradiction}
The two above conditional second moment bounds contradict (\ref{formula for L}, \ref{formula for epsilon1}), and the the contradiction assumption (\ref{contradiction assumption}).
To be more precise, combining (\ref{bound on Fab}) with (\ref{bound on Fab 2}) and (\ref{contradiction assumption}), we obtain 
$$\frac{R}{6} \leq \Expect_{\substack{a \in A \\ b \in B}} (|F_{ab}|^2|S) \leq Md^2\left(\frac{R}{c_N^2 L^{2/3} \theta} + \frac{2L^{2/3} \sigma^2}{c_N^2 \rho}\right) < Md^2\left(\frac{R}{c_N^2 L^{2/3} \theta} + \frac{2L^{2/3} \varepsilon_1 R}{c_N^2 \rho}\right).$$
Dividing both sides by $RM$ and applying (\ref{formula for L}, \ref{formula for epsilon1}), we obtain 
$$2 \cdot 10^{-8} < \frac{1}{48 \cdot 10^6} = \frac{1}{6M} \leq \frac{d^2}{c_N^2 L^{2/3}\theta} + \frac{2d^2 L^{2/3} \varepsilon_1}{c_N^2 \rho} \leq \frac{1}{10^{8}} + \frac{2}{10^9} = 1.2 \cdot 10^{-8}.$$
This is a contradiction that proves that $\sigma^2 \geq \varepsilon_1 R$, and so completes the proof of Proposition \ref{statistical iteration}.

\subsection{Proof of main theorem}
To prove Theorem \ref{main theorem} we iterate Proposition \ref{statistical iteration}.
For each $J$, we define
\begin{align*}
    E_J: V_{K - H(J)}(X) &\to \RR\\
    I &\mapsto \|F_J\|_{C_\theta(I)}.
\end{align*}
We endow $V_n(X)$ with the discrete measure induced by $\mu_X$, namely $\mu_X(\{I\}) = \mu_X(I)$, and $J$ with the restricted fractal measure $\mu_Y$. 

First suppose that $J \in V_K(Y)$. Then by the Cauchy--Schwarz inequality, it follows that
\begin{align*}
|\nabla F_J(x)| &= \frac{1}{\mu_Y(J)}\int_J i\partial_x \Psi_J(x, y) \exp(i(\Psi_J(x, y))) p(x, y) f(x, y) \\
&\qquad + \exp(i(\Psi_J(x, y))) \partial_x p(x, y) f(y) \dif \mu_Y(y) \\
&\le \frac{1}{\sqrt{\mu_Y(J)}}\left(\frac{\diam J}{h} \|\partial^2_{xy} \Phi\|_{C^0} \|f\|_{L^2(J)} \|p\|_{C^0} + \|\partial_x p\|_{C^0}\|f\|_{L^2(J)}\right)
\end{align*}
and
$$\|F_J\|_{C^0} \le \frac{\|p\|_{C^0}\|f\|_{L^2(J)}}{\sqrt{\mu_Y(J)}}.$$
Thus, 
\begin{equation}\label{base case}
E_J(I) = \|F_J\|_{C_\theta(I)} \leq \frac{\|p\|_{C^1} \|f\|_{L^2(J)}}{\sqrt{\mu_Y(J)}}.
\end{equation}
Taking $L^2$ norms of both sides of (\ref{base case}), we get 
\begin{equation}\label{integrated base case}
    \|E_J\|_{L^2}^2 \leq \frac{\|p\|_{C^1}^2 \mu_X(X)}{\mu_Y(J)} \|f\|_{L^2(J)}^2.
\end{equation}
If we take $L^2$ norms of both sides of (\ref{improved volume bound}), we get 
\begin{equation}\label{integrated inductive case}
\|E_J\|_{L^2}^2 \leq (1 - \varepsilon_1) \Expect_{b \in B} \|E_{J_b}\|_{L^2}^2.
\end{equation}
Inducting backwards on $H(J)$ with (\ref{integrated base case}) as base case and (\ref{integrated inductive case}) as inductive case, we conclude that if $J$ is a tile in $Y$ such that $H(J) = 0$, 
$$\|E_J\|_{L^2}^2 \leq \frac{\|p\|_{C^1}^2 \mu_X(X)}{\mu_Y(J)} (1 - \varepsilon_1)^K \|f\|_{L^2(J)}^2.$$
Summing both sides in $J$, we obtain
$$\|\mathcal B_h f\|_{L^2}^2 \lesssim \|p\|_{C^1}^2 \mu_X(X) \mu_Y(Y) (1 - \varepsilon_1)^K \|f\|_{L^2}^2.$$
We now can set 
$$\varepsilon_0 := \frac{\varepsilon_1}{6 \log L} \leq \frac{\log(1 - \varepsilon_1)^{-1}}{2 \log L}$$
and plug in $\theta$ in (\ref{formula for L}) to obtain (\ref{L}), (\ref{epsilon0}).
Then $(1 - \varepsilon_1)^{K/2} \leq h^{\varepsilon_0}$, so
$$\|\mathcal B_h\|_{L^2(\mu_Y) \to L^2(\mu_X)} \lesssim\|p\|_{C^1} \sqrt{\mu_X(X) \mu_Y(Y)} h^{\varepsilon_0}$$
which completes the proof of Theorem \ref{main theorem}.

\section{Applications}\label{applications}
\subsection{Classical fractal uncertainty principle}
We now prove Corollary \ref{FUP classic}, following \cite[Theorem 1, Remarks 1]{Dyatlov_2018}.

\begin{lemma}\label{regularity at scale h}
Let $(X, \mu)$ be $\delta$-regular on scales $[h, 1]$, $h > 0$, where $\delta \in [0, d]$, and $\mu$ is the $\delta$-dimensional Hausdorff measure.
Let $X_h := X + B_h$ and
$$\mu_h(A) := h^{\delta - d} |X \cap A|.$$
Then $(X_h, \mu_h)$ is $\delta$-regular on scales $[2h, 1]$ with constant
$$C_R(X_h) := 6^\delta |\Ball^d| C_R(X)^2.$$
\end{lemma}
\begin{proof}
Let $N = N_X(x, r, h)$ be the cardinality of a maximal $h$-separated set $X \cap B(x, r)$, for $x \in X$ and $r \geq 2h$.
By \cite[Lemma 7.4]{Dyatlov_2016}, we have 
$$C_R(X)^{-2} \frac{r^\delta}{h^\delta} \leq N_X(x, r, h) \leq C_R(X)^2 \left(1 + \frac{2r}{h}\right)^\delta.$$
If $\{x_1, \dots, x_N\}$ is such a maximal set, and $I_n := B(x_n, 2h)$, then $X \cap B(x, r) \subseteq \bigcup_{n=1}^N I_n$, so 
$$\mu_h(B(x, r)) \leq h^{\delta - d} \sum_{n=1}^N |I_n| \leq (2h)^\delta |\Ball^d| N \leq 2^\delta |\Ball^d| C_R(X)^2 (h + 2r)^\delta \leq C_R(X_h) r^\delta.$$
Conversely, if $J_n := B(x_n, h/2)$, then $J_n,J_m$ are disjoint, and $\bigcup_{n=1}^N J_n \subseteq X \cap B(x, r)$, so 
\begin{align*}
\mu_h(B(x, r)) &\geq \sum_{n=1}^N B(x_n, h/2) \geq N\frac{h^\delta}{2^\delta} \geq C_R(X)^{-2} 2^{-\delta} r^\delta \geq C_R(X_h)^{-1} r^\delta. \qedhere 
\end{align*}
\end{proof}

\begin{lemma}
    Let $(X, Y)$ be $\Phi$-nonorthogonal on scales $[h, 1]$, $h > 0$. Then $(X_h, Y_h)$ is $\Phi$-nonorthogonal on scales $[2h, 1]$ with constant $c_N(X_h, Y_h) := c_N(X, Y)/4$.
\end{lemma}
\begin{proof}
    Let $x_0 \in X_h$, $y_0 \in Y_h$, and $r_X, r_Y \geq 2h$; then there exist $\tilde x_0 \in X$ and $\tilde y_0 \in Y$ with
$$\max(|x_0 - \tilde x_0|, |y_0 - \tilde y_0|) \leq h.$$
Putting $\tilde r_X := r_X - h$ and $\tilde r_Y := r_Y - h$, we can find by $\Phi$-nonorthogonality of $(X, Y)$ points
$$x_1, x_2 \in X \cap B(\tilde x_0, \tilde r_X) \subseteq X \cap B(x_0, r_X)$$
and 
$$y_1, y_2 \in Y \cap B(\tilde y_0, \tilde r_Y) \subseteq Y \cap B(y_0, r_Y)$$
such that 
\begin{align*}
    |\Phi(x_1, y_1) - \Phi(x_1, y_2) - \Phi(x_2, y_1) + \Phi(x_2, y_2)| &\geq c_N(X) \tilde r_X \tilde r_Y \geq c_N(X_h) r_X r_Y. \qedhere
\end{align*}
\end{proof}

\begin{proof}[Proof of Corollary \ref{FUP classic}]
We introduce the Fourier integral operator
$$\mathcal B_h f(\xi) := \int_Y e^{ix \cdot \xi/h} f(x) \dif \mu_{Y, h}(x).$$
By the above lemmata, $(X_h, \mu_{X, h})$ is $\delta$-regular, $(Y_h, \mu_{Y,h})$ is $\delta'$-regular, and $(X_h, Y_h)$ is $\Phi$-nonorthogonal.
Thus by Theorem \ref{main theorem},\footnote{The fact that regularity and nonorthogonality only hold up to scale $2h$ cause us to incur a loss of a power of $2$, but this is irrelevant.} there exists $\varepsilon_0 > 0$ such that
\begin{align*}
    \|1_{X_h} \mathscr F_h 1_{Y_h}\|_{L^2 \to L^2} &= \frac{h^{d/2 - \delta}}{(2\pi)^{\delta/2}} \|\mathcal B_h\|_{L^2(\mu_{Y, h}) \to L^2(\mu_{X, h})} \lesssim h^{d/2 - \delta + \varepsilon_0}. \qedhere
\end{align*}
\end{proof}



\subsection{Convex cocompact hyperbolic manifolds}\label{s:covcocom} In this section we prove Theorem \ref{scattering theory}.
First we recall some preliminaries for convex cocompact hyperbolic manifolds. 

Let $\mathbf{H}^{d+1}$ be the $d+1$ dimensional hyperbolic space (with constant curvature $-1$). The orientation preserving isometry group is given by $G=SO(d+1,1)_0$. Let $K=SO(d+1)$ be a maximal compact subgroup, so that $\mathbf{H}^{d+1}=G/K$. We are interested in infinite volume hyperbolic manifolds given by $M=\Gamma\backslash G/K$ where $\Gamma\subset G$ is a convex cocompact Zariski dense torsion-free discrete subgroup.

Let $\mathfrak{o}=[\id]$ be the reference point in $\mathbf{H}^{d+1}$. The \textit{limit set} is defined as $\Lambda(\Gamma)=\lim \Gamma\mathfrak{o}\subset\partial_\infty(\mathbf{H}^{d+1})\subset\overline{\mathbf{H}^{d+1}}$. $\Gamma$ is called \textit{convex cocompact} if the convex core ${\rm Core}(M):=\Gamma\backslash {\rm Hull}(\Lambda(\Gamma))\subset M$ is compact. We say $\Gamma\subset G$ is \textit{Zariski dense} if the closure of $\Gamma$ is equal to $G$ with respect to the Zariski topology of $G$ viewed as an algebraic variety over $\RR$.
In the Poincar\'e upper half space model, the limit set $\Lambda(\Gamma)\subset \mathbf{R}^d\cup \{\infty\}$ is a compact set of dimension $\delta_\Gamma \in (0, d)$ (see \cite[\S2]{sarkar2021exponential}), and we may assume that $\Lambda(\Gamma)$ is a compact subset of $\RR^d$.

We recall the following non concentration property from Sarkar--Winter \cite[Proposition 6.6]{sarkar2021exponential}.

\begin{prop}\label{prop:ncp}
Let $\Gamma\subset G$ be a convex cocompact subgroup such that $\Gamma$ is Zariski dense in $G$. Then there exists $c_0>0$ so that for any $x\in\Lambda(\Gamma)\cap\mathbf{R}^d$, $\varepsilon\in (0,1)$ and $w\in \mathbf{R}^d$ with $|w|=1$, there exists $y\in \Lambda(\Gamma)\cap B(x,\varepsilon)$ so that
\begin{align}\label{NCP formula}
    |\langle y-x,w\rangle|>c_0\varepsilon.
\end{align}
\end{prop}
As a corollary we have
\begin{corollary}\label{c:nonortho_hyper}
Let $M$ be a convex cocompact hyperbolic $d+1$-fold such that $\Gamma$ is Zariski dense in $G$. Then for any $\Phi\in C^3(\RR^d\times \RR^d;\RR)$ such that $\partial_{xy}^2 \Phi(x,y)$ is nonvanishing, the pair $(\Lambda(\Gamma),\Lambda(\Gamma))$ is $\Phi$-non-orthogonal with some constant $c_N>0$ from scales $0$ to $1$.
\end{corollary}
\begin{proof}
By the mean value theorem, for $x_1, x_2\in B(x_0,r_X)$, $y_1, y_2\in B(y_0,r_Y)$, 
\begin{align*}
    &|\Phi(x_0, y_0) - \Phi(x_1, y_0) - \Phi(x_0, y_1) + \Phi(x_1, y_1)-\langle\partial_{xy}\Phi(x_0,y_0)(x_1-x_0), y_1-y_0\rangle|\\
    &\leq \|\Phi\|_{C^3}r_Xr_Y(r_X+r_Y).
\end{align*}
Let $H=\ker(\partial_{xy}^2\Phi(x_0,y_0))$ and $v$ be a unit normal vector to $H$ (if $H=\{0\}$, the we choose $v$ arbitrarily). By Proposition \ref{prop:ncp}, there exists $x_1\in \Lambda(\Gamma)\cap B(x_0,r_X)$ such that $|\langle x_1-x_0, v\rangle|>c_0r_X$. This would imply for some $c_1\in (0,1)$,
\begin{align*}
    |\partial_{xy}^2\Phi(x_0,y_0)(x_1-x_0)|>c_1c_0r_X.
\end{align*}
By Proposition \ref{prop:ncp} again, there exists $y_1\in \Lambda(\Gamma)\cap B(y_0, r_Y)$ such that
\begin{align*}
    |\langle\partial_{xy}^2\Phi(x_0,y_0)(x_1-x_0), y_1-y_0\rangle|>c_1c_0^2 r_X r_Y.
\end{align*}
Thus we may choose $r_X, r_Y\leq c_1c_0^2\|\Phi\|_{C^3}^{-1}/10$ so that
\begin{align*}
     |\Phi(x_0, y_0) - \Phi(x_1, y_0) - \Phi(x_0, y_1) + \Phi(x_1, y_1)|>\frac{c_1c_0^2}{2}r_Xr_Y,
\end{align*}
i.e. nonorthogonality holds with $c_N=\frac{c_1^3 c_0^6}{200 (1+\|\Phi\|_{C^3})^2}>0$.
\end{proof}

Theorem \ref{main theorem} and Lemma \ref{regularity at scale h} then implies $B_\chi(h):L^2(S^d)\to L^2(S^d)$ defined by
\begin{align*}
    B_\chi(h) u(x) = (2\pi h)^{-d/2} \int_{S^d}|x-y|^{2i/h}\chi(x,y) u(y)dy
\end{align*}
where $\chi(x,y)\in C_0^\infty(S^d\times S^d\setminus\{(x,x):x\in S^d\})$ satisfies the fractal uncertainty bound
\begin{align*}
    \|\mathbbm{1}_{\Lambda(\Gamma)(h)}B_\chi(h)\mathbbm{1}_{\Lambda(\Gamma)(h)}\|_{L^2(S^d)\to L^2(S^d)}\leq Ch^{\frac{d}{2}-\delta_\Gamma+\varepsilon_0}.
\end{align*}
By a covering argument as in \cite[Proposition 4.2]{Bourgain_2018}, we have for $\rho\in (0,1)$,
\begin{align*}
    \|\mathbbm{1}_{\Lambda(\Gamma)(h^\rho)}B_\chi(h)\mathbbm{1}_{\Lambda(\Gamma)(h^\rho)}\|_{L^2(S^d)\to L^2(S^d)}\leq Ch^{\frac{d}{2}-\delta_\Gamma+\varepsilon_0-2(1-\rho)}.
\end{align*}
Thus, $\Lambda(\Gamma)$ satisfies the fractal uncertainty principle with exponent $\beta=\frac{d}{2}-\delta_\Gamma+\varepsilon_0$ in the sense of \cite[Definition 1.1]{Dyatlov_2016}. Applying \cite[Theorem 3]{Dyatlov_2016}, we conclude the Laplacian on $M$ has only finitely many resonances in $\{\Im\lambda> \delta_\Gamma-\frac{d}{2}-\varepsilon_0+\varepsilon\}$ for any $\varepsilon>0$, proving Theorem \ref{scattering theory}.
 
\subsection{Computation of nonorthogonality constants}\label{how to compute}
The condition that $\Gamma\subset G$ being Zariski dense is qualitative, and so one needs to extract quantitative conditions, such as nonconcentration, from Zariski denseness by a compactness argument as in \cite{sarkar2021exponential}.
However, Qiuyu Ren has pointed out to us that for classical Schottky groups $\Gamma$ in $SO(3, 1)_0 = PSL(2, \CC)$, there is a simple and effective way to compute the nonorthogonality constant in Definition \ref{def:nonorthogonality}.
The key idea is to use the fact that M\"obius transformations are conformal maps and preserve circles in order to derive \eqref{NCP formula}.

We illustrate this by considering Schottky groups of genus $2$.
Let $D_1,D_2,D_3,D_4$ be four disjoint closed disks in $\mathbf{CP}^1=\partial \mathbf{H}^3$, let $\gamma_1,\gamma_2\in PSL(2,\CC)$ such that
\begin{align*}
    \gamma_1(\overline{D_3^c})=D_1,\quad \gamma_2(\overline{D_4^c})=D_2,\quad\gamma_3=\gamma_1^{-1},\quad\gamma_4=\gamma_2^{-1}.
\end{align*}
Let $\Gamma=\langle \gamma_1,\gamma_2\rangle$ be the free group generated by $\gamma_1$ and $\gamma_2$.
Thus, $\Gamma$ is a Schottky group of genus $2$.

Given vectors $v, w \in \RR^2$, let $\angle (v, w)$ denote the angle between $v, w$. (We identify $\mathbf{CP}^1 \setminus \{\infty\}$ with $\RR^2$, and we may assume that the $D_i$ do not contain $\infty$.)
We will choose the disks $D_1,D_2,D_3,D_4$ such that
 \begin{align}\label{condition:disk}
     \text{No circle (line) passes though all the four disks}.
 \end{align}
The circle taken here is not necessarily a great circle.

Let $\bar{a}\equiv a+2\mod 4$ for $a\in\Acal=\{1,2,3,4\}$, so that $\bar{1}=3$, $\bar{2}=4$. The limit set $\Lambda(\Gamma)$ is given by the Cantor-like procedure
\begin{align*}
    \Lambda(\Gamma)=\bigcap\limits_{n=1}^\infty\bigsqcup\limits_{\mathbf{a}\in \mathcal{W}^n} D_{\mathbf{a}},\quad \mathcal{W}^n=\{a_1a_2\cdots a_n\in\Acal^n:\overline{a_i}\neq a_{i+1}\},\quad D_{\mathbf a}=\gamma_{a_1}(\gamma_{a_2}\cdots(\gamma_{a_{n-1}}(D_{a_n}))).
\end{align*}

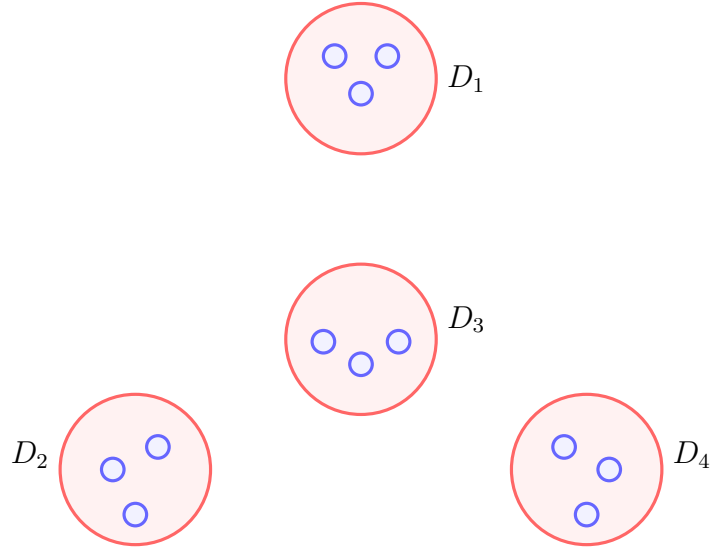
\begin{figure}
    \centering
\begin{tikzpicture}
\filldraw[color=red!60, fill=red!5, very thick](0,5.2) circle (1);
\node at (1.4,5.2) {$D_1$};

\filldraw[color=red!60, fill=red!5, very thick](-3,0) circle (1);
\node at (1.4,2) {$D_3$};

\filldraw[color=red!60, fill=red!5, very thick](3,0) circle (1);
\node at (-4.4,0.2) {$D_2$};

\filldraw[color=red!60, fill=red!5, very thick](0,1.732) circle (1);
\node at (4.4,0.2) {$D_4$};

\filldraw[color=blue!60, fill=blue!5, very thick](0,1.4) circle (0.15);
\filldraw[color=blue!60, fill=blue!5, very thick](0.5,1.7) circle (0.15);
\filldraw[color=blue!60, fill=blue!5, very thick](-0.5,1.7) circle (0.15);

\filldraw[color=blue!60, fill=blue!5, very thick](0,5) circle (0.15);
\filldraw[color=blue!60, fill=blue!5, very thick](0.35,5.5) circle (0.15);
\filldraw[color=blue!60, fill=blue!5, very thick](-0.35,5.5) circle (0.15);

\filldraw[color=blue!60, fill=blue!5, very thick](-3.3,0) circle (0.15);
\filldraw[color=blue!60, fill=blue!5, very thick](-2.7,0.3) circle (0.15);
\filldraw[color=blue!60, fill=blue!5, very thick](-3,-0.6) circle (0.15);

\filldraw[color=blue!60, fill=blue!5, very thick](3.3,0) circle (0.15);
\filldraw[color=blue!60, fill=blue!5, very thick](2.7,0.3) circle (0.15);
\filldraw[color=blue!60, fill=blue!5, very thick](3,-0.6) circle (0.15);

\end{tikzpicture}
\caption{Iteration of disks under a Schottky group}
\label{schottky}
\end{figure}
The nonorthogonality condition \eqref{nonorthogonality formula} follows from the nonconcentration property \eqref{NCP formula}. Thus it suffices to find absolute constants $0 < c_1 < 1$ and $\kappa = \kappa(\Gamma) > 0$ such that for each $x \in \Lambda(\Gamma)$, $\epsilon > 0$, and unit vector $w \in \mathbb{R}^2$, an element $y \in \Lambda(\Gamma)\cap B(x,\epsilon)\setminus B(x, c_1\epsilon)$ such that
$$|\cos \angle (x - y, w)| \geq \kappa.$$
Suppose $x\in D_{\mathbf{a}}=D_{\mathbf{a}_0 b}$ and $B(x,\epsilon)$ is roughly of the size of $D_{\mathbf{a}_0}$. Then there are two other disks in $D_{\mathbf{a}_0}$, which we call $D_{\mathbf{a}_0c}$ and $D_{\mathbf{a}_0d}$. By condition \eqref{condition:disk} and conformal invariance of the action of $\Gamma$, we know that for any $y_c\in D_{\mathbf{a}_0c}\cap \Lambda(\Gamma)$ and $y_d\in D_{\mathbf{a}_0d}\cap\Lambda(\Gamma)$, 
\begin{align}\label{insidecircle}
    \text{ the circle passing through } x, y_c, y_d \text{ lies inside } D_{\mathbf{a}_0}.
\end{align}
A M\"obius transformation preserving the unit disk is a composition of rotation and the map
\begin{align*}
    z\mapsto \frac{a-z}{1-\bar{a}z}
\end{align*}
A simple computation shows the angles of the triangle $\Delta(x,y_c,y_d)$ are uniformly lower bounded under conformal maps preserving $D_{\mathbf{a}_0}$ if we assume \eqref{insidecircle}. This implies that
\begin{align*}
    \theta<\angle (y_c-x,y_d-x)<\pi-\theta
\end{align*}
for some constant $\theta$ depending on the initial angles between $\gamma_a(D_b),a\neq \bar{b}$.
Thus, by the pigeonhole principle,
$$\max(|\cos \angle (y_c - x, w)|, |\cos \angle(y_d - x, w)|) \geq \cos\left(\frac{\pi - \theta}{2}\right).$$

If we assume moreover
\begin{equation}\label{condition:disk2}
    \begin{aligned}
    \text{For any } b \neq \bar{a} \neq c, \text{ there exists } a'\neq a, b'\neq \overline{a'} \text{ such that }\\
    \text{no 
    circle passes through } \gamma_a(D_b),\gamma_a(D_c),\gamma_{a'}(D_{b'}) \text{ and } D_{\bar{a}}
\end{aligned}
\end{equation}
(which can be achieved if we choose the disks $D_a$ to be small and with generic centers), then we can derive a lower bound on $c_1$ in  a similar way. To be more precise, let $x\in D_{\mathbf{a}}=D_{\mathbf{a}_0b}=S_{\mathbf{a}_1ab}$ as before, then by assumption \ref{condition:disk2}, there exists $a'\neq a$ and $b'\neq \overline{a'}$ such that \begin{align}
    \text{ The circle passing through }D_{\mathbf{a}_0b}, D_{\mathbf{a}_0c} \text { and } D_{\mathbf{a}_1a'b'} \text{ lies inside } D_{\mathbf{a}_1}.
\end{align}
In particular, for any $y_{a'b'}\in D_{\mathbf{a_1}a'b'} $, the angles of the triangle $\Delta(x,y_c,y_{a'b'})$ are lower bounded. This in particular implies that the length of $\overline{xy_c}$ is comparable to the length of $\overline{y_cy_{a'b'}}$, which by the previous step is comparable with the size of $D_{\mathbf{a}_0}$. This allows us to compute a lower bound of $c_1$.

If one runs this procedure carefully, then it would be possible to compute an explicit nonorthogonality constant in terms of the angles between the disks $\gamma_a(D_b)$ in the initial step and the uniform constants in doing conformal transformations.

We do not bother to do the computation here, but we include Figure \ref{schottky} to indicate how the procedure works. Conformal invariance ensures us that the small blue disks always have an angle that lies in $[\theta, \pi - \theta]$.

While one needs to compute the above parameters $\kappa, \theta$ for any given Zariski dense classical Schottky group $\Gamma$, we claim that this is always possible in principle, at least after passing to a finer scale.
We say that a pair of words $\mathbf a, \mathbf b \in \mathcal W^n$, $n \in \NN \cup \{+\infty\}$, is $\varepsilon$-\dfn{separated} if their weighted Hamming distance satisfies
$$\sum_{i = 1}^n \frac{1_{a_i \neq b_i}}{2^i} \geq \varepsilon.$$

\begin{lemma}\label{lem:zariski}
Let $\Gamma$ be a classical Schottky group which is Zariski dense in $PSL(2, \CC)$.
For every $\varepsilon > 0$ there exists $N \in \NN$ such that for every $n \geq N$ and every triple of words $\mathbf a^n, \mathbf b^n, \mathbf c^n \in \mathcal W^n$ which are pairwise $\varepsilon$-separated, there exists $\mathbf d^n \in \mathcal W^n$ such that for every circle $X$ which meets all three disks $D_{\mathbf a^n}, D_{\mathbf b^n}, D_{\mathbf c^n}$, $X$ does not meet $D_{\mathbf d^n}$.
\end{lemma}
\begin{proof}
We first prove an analogous result for the set of infinite words $\mathcal W^\infty$, and then reduce the finite case to the infinite case.
To formulate it, let $x_{\mathbf a}$ be the unique point in $\lim_n D_{a_1 \cdots a_n}$ (so $\mathbf a \mapsto x_{\mathbf a}$ is a homeomorphism $\mathcal W^\infty \to \Lambda(\Gamma)$ where $\mathcal W^\infty$ is given the product topology).

Let $\mathbf a, \mathbf b, \mathbf c \in \mathcal W^\infty$ be distinct.
Then there is a unique circle $X_{\mathbf a \mathbf b \mathbf c} \subset \mathbf{CP}^1$ passing through $x_{\mathbf a}, x_{\mathbf b}, x_{\mathbf c}$.
We claim that there exists $\mathbf d \in \mathcal W^\infty$ such that $x_{\mathbf d} \notin X_{\mathbf a \mathbf b \mathbf c}$. Otherwise $\Lambda(\Gamma)$ is contained in a circle, which contradicts Proposition \ref{prop:ncp}.

We now address the finite case.
Suppose that the lemma fails on some $\mathbf a^n, \mathbf b^n, \mathbf c^n \in \mathcal W^n$ for each $n \in \NN$ which are $\varepsilon$-separated, so for every $\mathbf d^n \in \mathcal W^n$ there exists a circle $X(\mathbf d^n)$ which meets all disks $D_{\mathbf a^n}, D_{\mathbf b^n}, D_{\mathbf c^n}, D_{\mathbf d^n}$.
Let $\mathbf a, \mathbf b, \mathbf c \in \mathcal W^\infty$ be the limits of $\mathbf a^n$, et cetra, and let $\mathbf d \in \mathcal W^\infty$ be given.
Then $\mathbf d = \lim_n \mathbf d^n$ for some sequence $\mathbf d^n \in \mathcal W^n$, and we can define $X := \lim_n X(\mathbf d^n)$ in Hausdorff distance.
Then, $x_{\mathbf a}, x_{\mathbf b}, x_{\mathbf c}, x_{\mathbf d} \in X$, and $\mathbf a, \mathbf b, \mathbf c$ are $\varepsilon$-separated, hence distinct.
Moreover, $X$ is the limit of circles in $\mathbf{CP}^1$ whose radii are bounded from below (by $\varepsilon$-separation), so $X$ is a circle, hence $X = X_{\mathbf a \mathbf b \mathbf c}$.
This contradicts the infinite case.
\end{proof}

Assuming Lemma \ref{lem:zariski}, for $D_{\mathbf{a}}=D_{a_1\cdots a_{2n}}$, we can find $\mathbf{b}, \mathbf{c} \in\mathcal{W}^{2n}$ such that any circle passing through $D_{\mathbf{a}}, D_{\mathbf{b}}$ and $D_{\mathbf{c}}$ lies in the disk $D_{a_1}$. This is because given $D_{a_1\cdots a_{2n}}$ and $\overline{D_{a_1}^c}$, we have
\begin{align*}
    \gamma_{\bar{a}_n}\cdots\gamma_{\bar{a}_2}\gamma_{\bar{a}_1}(D_{a_1\cdots a_{2n}})=D_{a_{n+1}\cdots a_{2n}},\quad \gamma_{\bar{a}_n}\cdots\gamma_{\bar{a}_2}\gamma_{\bar{a}_1}(\overline{D_{a_1}^c})=D_{\bar{a}_n\cdots\bar{a}_2\bar{a}_1}.
\end{align*}
By Lemma \ref{lem:zariski}, there exists $\mathbf{b}_0,\mathbf{c}_0\in\mathcal{W}^n$ such that no circle passes through $D_{a_{n+1}\cdots a_{2n}}$, $D_{\bar{a}_n\cdots\bar{a}_2\bar{a}_1}$, $D_{\mathbf{b}_0}$ and $D_{\mathbf{c}_0}$. Applying $\gamma_{a_1}\cdots \gamma_{a_n}$, we conclude any circle passing through
\begin{align*}
    D_{a_1\cdots a_{2n}}, D_{a_1\cdots a_n\mathbf{b}_0}, D_{a_1\cdots a_n \mathbf{c}_0}
\end{align*}
lies inside $D_{a_1}$ (there might be cancellations for the words $a_1\cdots a_n\mathbf{b}_0$ and $a_1\cdots a_n \mathbf{c}_0$ but one can always pass to a smaller disk). This allows us to compute the angle $\theta$ as before for general Zariski dense classical Schottky groups.

\printbibliography

\end{document}